\documentclass[11pt]{article}
\usepackage{amsfonts,amsmath,amssymb,amsthm}
\usepackage{tikz,color,cite}

\newtheorem{theorem}{Theorem}[section]

\newtheorem{problem}[theorem]{Problem}

\newtheorem{lemma}[theorem]{Lemma}

\begin{document}

\title{Embeddings into almost self-centered graphs of  given radius}

\author{
Kexiang Xu $^{a}$
\and
Haiqiong Liu $^{a}$
\and
Kinkar Ch. Das $^{b}$
\and
Sandi Klav\v zar $^{c,d,e}$
}

\date{}

\maketitle

\begin{center}
$^a$ College of Science, Nanjing University of Aeronautics \& Astronautics\\
Nanjing, Jiangsu, 210016, PR China \\
{\tt kexxu1221@126.com} (K. Xu) \\
{\tt lhqstu@163.com} (H. Liu)
\medskip

$^b$ Department of Mathematics, Sungkyunkwan University\\
Suwon 440-746, Republic of Korea \\
{\tt kinkardas2003@gmail.com}
\medskip

$^c$ Faculty of Mathematics and Physics, University of Ljubljana, Slovenia\\
{\tt sandi.klavzar@fmf.uni-lj.si}
\medskip

$^d$ Faculty of Natural Sciences and Mathematics, University of Maribor, Slovenia\\
\medskip

$^e$ Institute of Mathematics, Physics and Mechanics, Ljubljana, Slovenia
\end{center}


\begin{abstract}
A graph is almost self-centered (ASC) if all but two of its vertices are central. An almost self-centered graph with radius $r$ is called an $r$-ASC graph. The $r$-ASC index $\theta_r(G)$ of a graph $G$ is the minimum number of vertices needed to be added to $G$ such that an $r$-ASC graph is obtained that contains $G$ as an induced subgraph.  It is proved that $\theta_r(G)\le 2r$ holds for any graph $G$ and any $r\ge 2$ which improves the earlier known bound $\theta_r(G)\le 2r+1$. It is further proved that $\theta_r(G)\le 2r-1$ holds if  $r\geq 3$ and $G$ is of order at least $2$. The $3$-ASC index of complete graphs is determined. It is proved that $\theta_3(G)\in \{3,4\}$ if $G$ has diameter $2$  and for several classes of graphs of diameter $2$ the exact value  of the $3$-ASC index is  obtained. For instance, if a graph $G$ of diameter $2$ does not contain a diametrical triple, then $\theta_3(G) = 4$. The $3$-ASC index of paths  of order $n\geq 1$, cycles  of order $n\geq 3$, and trees of order $n\geq 10$ and diameter $n-2$ are also determined, respectively, and several open problems proposed.
\end{abstract}

\noindent {\bf Key words:} eccentricity; diameter; almost self-centered graph;  graph of diameter $2$; ASC index

\medskip\noindent
{\bf  Math. Subj. Class. (2010)}: 05C12, 05C05, 05C75

 \section{Introduction and preliminaries}

Almost self-centered graphs (ASC graphs for short) were introduced in~\cite{KNW2011} as the graphs  which contain exactly two non-central vertices. If the radius of an ASC graph $G$ is $r$, then $G$ is more precisely called an $r$-ASC graph. The introduction of these graphs was in particular motivated with  network designs in which two (expensive) resources that need to be far away due to interference reasons must be installed.  Numerous additional situations appear in which two specific locations are desired, for instance in trees~\cite{hl-00,ww-09}, block graphs~\cite{alcon-2017, cheng-2014}, and interval graphs~\cite{hong-2014}. While in the first two of these papers 2-centers are studied in trees, the results could be used also in the general case. For a given graph $G$, one could
first select a suitable subtree (spanning or otherwise), then determine a 2-center, and finally lift the obtained solution to $G$.

Graphs dual to the almost self-centered graphs were studied in~\cite{KNWL2014} and named almost-peripheral graphs (AP graphs for short). Very recently, a measure of non-self-centrality was introduced in~\cite{xu-2016+}, where ASC graphs and AP graphs, along with a newly defined weakly AP graphs, play a significant role as extremal graphs in studies of this new measure.

Among other results it was proved in~\cite{KNW2011} that for any connected graph $G$ and any $r\ge 2$ there exists an $r$-ASC graph which contains $G$ as an induced subgraph. Consequently, the {\em $r$-ASC index} $\theta_r(G)$ of $G$ was introduced as the minimum number of vertices needed to add to $G$ in order to obtain an $r$-ASC graph that contains $G$ as an induced subgraph. It was proved that for any connected graph $G$ we have $\theta_2(G)\le 2$ and that $\theta_r(G)\le 2r+1$ holds for any $r\ge 3$. Moreover, the $2$-ASC index was determined exactly for complete bipartite graphs, paths and cycles. The study of ASC graphs was continued in~\cite{BBCK2012} where ASC graphs were characterized among chordal graphs and among median graphs, and in~\cite{KLSX-2017+} where existing constructions of ASC graphs were revisited and new constructions of such graphs  were developed. See also \cite{Gu-1993} for some related embeddings.

We note in passing that graph eccentricity is frequently applied in chemical graph theory via the so-called eccentricity-based topological indices. Although these indices are pure graph theory concepts, they have a substantial use in theoretical chemistry, cf.~\cite{SGM1997}. The eccentricity-based topological indices include eccentric
distance sum~\cite{IYF2011}, Zagreb eccentricity indices~\cite{DLG2013}, and connective eccentricity index~\cite{XDL2016,YQTF2014}. See also  \cite{DN2016} for some recent results on the general distance-based topological indices and a recent survey~\cite{XLDGF2014} on extremal results on general distance-based topological indices. Moreover, the applications of eccentricity to networks and location theory can be seen in~\cite{krnc-2015} and in~\cite{puerto-2008}, respectively.

This paper is organized as follows. In the rest of this section definitions and notations needed are given.  In the following section we first prove that if $G$ is a graph of order $n\ge 1$ and $r\ge 2$, then $\theta_r(G)\le 2r$, thus improving the so-far best known bound $\theta_r(G)\le 2r+1$. In the main result of the section we further strengthen the bound by proving that $\theta_r(G) \le 2r-1$ holds if $r\ge 3$ and $G$ has at least two vertices. In the rest of the paper we then consider the $3$-ASC index. The obtained results and (new) techniques involved indicate that to determine the $3$-ASC index of a graph is generally a much more complex task than to determine its $2$-ASC index. In Section~\ref{sec:small-diam}, we determine the $3$-ASC index of complete graphs and  investigate the index on graphs of diameter $2$.  We prove that if $G$ is such a graph, then either $\theta_3(G) = 3$ or $\theta_3(G) = 4$. For several classes of graphs of diameter $2$ the exact value is determined. In Section~\ref{sec:large-diam} we turn our attention to graphs with large diameter and determine the $3$-ASC index of paths, cycles, and trees of order $n$ and diameter $n-2$. We conclude the paper with several open problems.

We only consider finite, undirected, simple graphs throughout this paper. The \textit{degree}  ${\rm deg}_G(v)$ of a vertex $v$ of a graph $G$ is the cardinality of the neighborhood $N_G(v)$ of $v$. The closed neighborhood $N_G[v]$ of $v$ is $N_G(v)\cup \{v\}$. The maximum and minimum degree of $G$  are denoted by $\Delta(G)$ and $\delta(G)$, respectively. A vertex of degree $1$ is called a \textit{pendant vertex}, the edge incident with  a pendant vertex is called a \textit{pendant edge}.

The {\em distance} $d_G(u,v)$ between vertices $u$ and $v$ is the length (that is, the number of edges) of a shortest $(u,v)$-path. The \textit{eccentricity} ${\rm ecc}_G(u)$ of a vertex $u$ is the maximum distance from $u$ to other vertices in $G$. The \textit{diameter} ${\rm diam}(G)$ of $G$ is the maximum eccentricity among its vertices and the \textit{radius} ${\rm rad}(G)$ is the minimum eccentricity of its vertices. See~\cite{MW-2000} for some properties of the vertex eccentricity in a graph.  A vertex $v$ is an \textit{eccentric vertex} of a vertex $u$ if $d_G(u,v)={\rm ecc}_G(u)$. The \textit{eccentric set} ${\rm Ecc}_G(u)$ of $u$ is the set of all its eccentric vertices. A vertex $u$ of $G$ is called a \textit{central vertex} if ${\rm ecc}_G(u)={\rm rad}(G)$ and is called a \textit{diametrical} vertex if ${\rm ecc}_G(u)={\rm diam}(G)$. The {\em center} $C(G)$ of $G$ consists of all its central vertices, while the {\em periphery} $P(G)$ of $G$ contains all its diametrical vertices.  A graph $G$ is {\em self-centered} (SC graph for short) if $C(G) = V(G)$; we refer to the survey on SC graphs~\cite{buckley-1989} and to their  recent application in~\cite{madalloni-2016}. If ${\rm rad}(G) = r$, then we also say that $G$ is an $r$-SC graph. Vertices $u$ and $v$ with $d_G(u,v)={\rm diam}(G)$ will be referred to as a {\em diametrical pair}. Similarly, a triple of vertices that are pairwise at distance ${\rm diam}(G)$ is called a {\em diametrical triple}. When no ambiguity  occurs, the subscript $G$ will be omitted in the notations from the  last two paragraphs.

If $G$ is a graph, then $\overline{G}$ denotes the complement of $G$. The {\em disjoint union} of (vertex-disjoint) graphs $G_1$ and $G_2$ will be denoted with $G_1\cup G_2$, while the \textit{join} of $G_1$ and $G_2$ will be denoted by $G_1\oplus G_2$. Recall that $G_1\oplus G_2$ is obtained from $G_1\cup G_2$ by adding an edge between any vertex of $G_1$ and any vertex of $G_2$. If $X\subseteq V(G)$, then $G[X]$ denotes the subgraph of $G$ induced by $X$. Finally, throughout this paper we use $P_n$, $C_n$, and $K_n$ to denote the path graph, the cycle graph, and the complete graph on $n$ vertices, respectively.

\section{Improving upper bounds on the $r$-ASC index}
\label{sec:improve-upper-bounds}

It was proved in~\cite[Corollary 4.1]{KNW2011} that $\theta_r(G)\leq 2r+1$ holds for any graph $G$. In this section we first sharpen this result as follows.  Interestingly, the  below construction which gives a better upper bound is simpler than the construction from~\cite{KNW2011}.

\begin{theorem}
\label{thm:constuction}
If $G$ is a graph of order $n\ge 1$ and $r\ge 2$, then $\theta_r(G)\le 2r$.
\end{theorem}

\begin{proof}
Let $\widehat{G}_r$ be the graph obtained from $G$ as follows. Add three new vertices $w$, $x_1$, $y_1$, and add an edge between each of them and any vertex of $G$, so that in this way $3n$ edges are added. Further, add vertices $x_2, \ldots, x_{r-1}$, $y_2, \ldots, y_{r-1}$, and $w'$. Finally,  add edges $x_1x_2, \ldots, x_{r-2}x_{r-1}$, $y_1y_2, \ldots, y_{r-2}y_{r-1}$, $x_{r-1}w'$, and $y_{r-1}w'$. The construction is illustrated in Fig.~\ref{fig:2r-construction}.

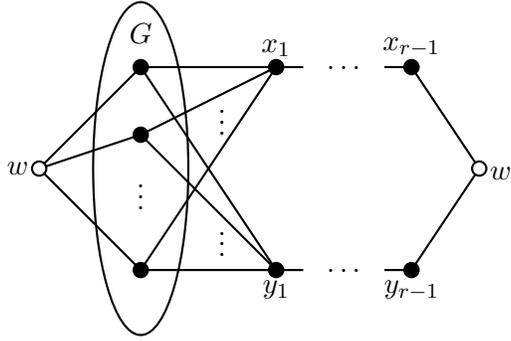
\begin{figure}[ht!]
\begin{center}
\begin{tikzpicture}[scale=0.9,style=thick]
\def\vr{3pt}
\path (-0.5,1.5) coordinate (a);
\path (1,0) coordinate (d);
\path (1,2) coordinate (c);
\path (1,3) coordinate (b);
\path (3,0) coordinate (u);
\path (3,3) coordinate (x);
\path (5,0) coordinate (v);
\path (5,3) coordinate (y);
\path (6,1.5) coordinate (z);
\draw (a) -- (b); \draw (a) -- (c); \draw (a) -- (d);
\draw (b) -- (x); \draw (d) -- (u);
\draw (y) -- (z) -- (v);
\draw (x) -- (c);
\draw (x) -- (d);
\draw (u) -- (b);
\draw (u) -- (c);
\draw (x) -- (3.4,3);
\draw (u) -- (3.4,0);
\draw (4.6,3) -- (y);
\draw (4.6,0) -- (v);
\draw (a)  [fill=white] circle (\vr);
\draw (b)  [fill=black] circle (\vr);
\draw (c)  [fill=black] circle (\vr);
\draw (d)  [fill=black] circle (\vr);
\draw (x)  [fill=black] circle (\vr);
\draw (y)  [fill=black] circle (\vr);
\draw (u)  [fill=black] circle (\vr);
\draw (v)  [fill=black] circle (\vr);
\draw (z)  [fill=white] circle (\vr);
\draw (1,1.2) node {$\vdots$};
\draw (2.2,2.3) node {$\vdots$};
\draw (2.2,0.5) node {$\vdots$};
\draw (4,3) node {$\ldots$};
\draw (4,0) node {$\ldots$};
\draw (1,1.5) ellipse (20pt and 70pt);
\draw (1,3.5) node {$G$};
\draw [left] (a) node {$w$};
\draw [right] (z) node {$w'$};
\draw [above] (x) node {$x_1$};
\draw [above] (y) node {$x_{r-1}$};
\draw [below] (u) node {$y_1$};
\draw [below] (v) node {$y_{r-1}$};
\end{tikzpicture}
\end{center}
\caption{Graph $\widehat{G}_r$}
\label{fig:2r-construction}
\end{figure}

It is straightforward to verify that $d_{\widehat{G}_r}(w,w') = r+1$ and ${\rm ecc}_{\widehat{G}_r}(w) = {\rm ecc}_{\widehat{G}_r}(w')= r+1$.  If $a\in V(G)$, then the sequence of vertices $a,x_1,\ldots, x_{r-1}, w', y_{r-1}, \ldots, y_1, a$ induces an isometric cycle of length $2r$. Using this fact if follows by a simple inspection that for any vertex $x\in V(\widehat{G}_r)$, $x\ne w,w'$, we have ${\rm ecc}_{\widehat{G}_r}(x) = r$. We conclude that $\widehat{G}_r$ is an $r$-ASC graph and consequently $\theta_r(G)\le 2r$.
\end{proof}

Next we characterize the graphs $G$ with $\theta_r(G)=2r$ for any $r\ge 3$. Before doing it, we need the following result.

\begin{lemma}\label{conn-r-ASC} Let $G$ be a connected graph of order $n\geq 2$  and $r\geq 3$. Then $\theta_r(G)\leq 2r-1$.
\end{lemma}

\begin{proof}
The graph $C_{2r}'$, which is obtained from $C_{2r}$ by attaching a pendant vertex to it, is an $r$-ASC graph. It follows that $\theta_r(K_2)\leq 2r-1$. In the following we may assume that $n\ge 3$. It suffices to construct a graph $H$ of order at most $n+2r-1$ as an $r$-ASC embedding graph of $G$.

 Now we construct a graph $H$ with vertex set $V(G)\cup\{x_1,\ldots,x_{r-1},y_1,\ldots,y_{r-1},w\}$. Set $V_0=\{x_1,\ldots,x_{r-1},y_1,\ldots,y_{r-1},w\}$. We choose two arbitrary vertices $x$ and $y$ with $xy\in E(G)$. In $H$, the adjacency relation of $V_0$ is $x_1-x_2-\cdots-x_{r-1}-y_{r-1}-\cdots-y_1-w$ and  $x,y$ are joined with $x_1$ and $y_1$, respectively.  Let $N_2(x)$ and $N_2(y)$ be the set of vertices in $G$ at distance $2$ from $x$ and $y$, respectively. For any vertex $v\in N(x)$, we join $v$ with $y_1$ in $H$. And we join $v$ with $x_1$ in $H$ for any vertex $v\in N(y)\setminus N(x)$.  Set $V^*(G)=V(G)\setminus \left(N(x)\cup N(y)\right)$ and note that $x,y\notin V^*(G)$. Next we consider the positions of all vertices (if any) from $V^*(G)$ in $H$. For any vertex $v\in V^*(G)$, if $v\in N_2(y)$, $v$ is joined with $y_1$ in $H$; if $v\in N_2(x)\setminus N_2(y)$, then $vy_2\in E(H)$;  while if $z\in V^*(G) \setminus \left(N_2(x)\cup N_2(y)\right)$, then $zy_1,zy_3\in E(H)$ for $r\geq 4$ or $zy_1,zx_2\in E(H)$ for $r=3$.

 This construction is illustrated in Fig.~\ref{fig:at-most-2r-1}, where a thick edge from a vertex to a thick set indicates that the vertex is adjacent to every vertex in the set. The dashed thick line from $x_1$ indicates that $x_1$ is adjacent to every vertex in $N(y)\setminus N(x)$, while the dashed line from  $y_2$ says that $y_2$ is adjacent to every vertex in $N_2(x)\setminus N_2(y)$. The figure is drawn for the cases when $r\ge 4$. If $r=3$ then one only needs to replace the edge $zy_3$ (that is, from a typical vertex $z$ from $V^*(G) \setminus \left(N_2(x)\cup N_2(y)\right)$ to $y_3$ with the edge $zx_2$. Note that if $r=3$ then the vertices $y_3$ and $x_3$ do not exist, so that we can imagine that in the case $r=3$ the vertex $y_3$ is identified with $x_2$ (and $x_3$ with $y_2$) and then the edge $zy_3$ is just the edge $zx_2$.

\begin{figure}[ht!]
\begin{center}
\begin{tikzpicture}[scale=1.3,style=thick]
\def\vr{3pt}
\path (0,0) coordinate (x);
\path (1,0) coordinate (y);
\path (-4,4) coordinate (y1);
\path (-4.5,5) coordinate (w);
\path (-3,5.7) coordinate (y2);
\path (-2,5.7) coordinate (y3);
\path (0,5.7) coordinate (yr-1);
\path (1,5.7) coordinate (xr-1);
\path (5,4) coordinate (x1);
\path (4,5.7) coordinate (x2);
\path (3,5.7) coordinate (x3);
\path (-2,4) coordinate (z);
\draw (x) -- (y);
\draw (w) -- (y1) -- (y2) -- (y3);
\draw (x3) -- (x2) --(x1);
\draw (yr-1) -- (xr-1);
\draw (x) .. controls (2,-1.5) and (4,1) .. (x1);
\draw (y) .. controls (-1,-1.5) and (-3,1) .. (y1);
\draw (x) -- (-1.3,1.3);
\draw (x) -- (0.3,1.5);
\draw (y) -- (0.7,1.5);
\draw (y) -- (2.3,1.3);
\draw (-1,1.7) -- (-1,3.5);
\draw (2,1.7) -- (2,3.5);
\draw (-0.5,1.8) -- (0.3,3.5);
\draw (1.5,1.8) -- (0.7,3.5);
\draw[line width=1mm] (-1.3,1.7) -- (y1);
\draw[line width=1mm, dashed] (2.3,1.7) -- (x1);
\draw[line width=1mm] (1.4,3.8) .. controls (0,5) and (-1.2,5) .. (y1);
\draw[line width=1mm, dashed] (-0.4,3.8) -- (y2);
\draw (y1) -- (z) -- (y3);
\draw (y3) -- (-1.6,5.7);
\draw (yr-1) -- (-0.4,5.7);
\draw (x3) -- (2.6,5.7);
\draw (xr-1) -- (1.4,5.7);
\draw (x)  [fill=white] circle (\vr);
\draw (y)  [fill=white] circle (\vr);
\draw (w)  [fill=white] circle (\vr);
\draw (x1)  [fill=white] circle (\vr);
\draw (x2)  [fill=white] circle (\vr);
\draw (x3)  [fill=white] circle (\vr);
\draw (y1)  [fill=white] circle (\vr);
\draw (y2)  [fill=white] circle (\vr);
\draw (y3)  [fill=white] circle (\vr);
\draw (xr-1)  [fill=white] circle (\vr);
\draw (yr-1)  [fill=white] circle (\vr);
\draw (z)  [fill=white] circle (\vr);
\draw[line width=0.7mm] (-0.5,1.5) ellipse (40pt and 15pt);
\draw[line width=0.3mm] (1.5,1.5) ellipse (40pt and 15pt);
\draw[line width=0.3mm] (-0.5,3.5) ellipse (40pt and 15pt);
\draw[line width=0.7mm] (1.5,3.5) ellipse (40pt and 15pt);
\draw[line width=0.3mm] (0.5,4.0) ellipse (90pt and 35pt);
\draw[left] (x)++(-0.1,0) node {$x$};
\draw[right] (y)++(0.1,0) node {$y$};
\draw[left] (w)++(-0.1,0) node {$w$};
\draw[left] (y1)++(-0.1,0) node {$y_1$};
\draw[right] (x1)++(0.1,0) node {$x_1$};
\draw[above] (y2)++(0,0.1) node {$y_2$};
\draw[above] (y3)++(0,0.1) node {$y_3$};
\draw[above] (x2)++(0,0.1) node {$x_2$};
\draw[above] (x3)++(0,0.1) node {$x_3$};
\draw[above] (xr-1)++(0,0.1) node {$x_{r-1}$};
\draw[above] (yr-1)++(0,0.1) node {$y_{r-1}$};
\draw[right] (z)++(0.1,0) node {$z$};
\draw[right] (y3)++(0.7,0) node {$\cdots$};
\draw[right] (xr-1)++(0.7,0) node {$\cdots$};
\draw (-0.3,1.5) node {$N(x)$};
\draw (1.3,1.5) node {$N(y)$};
\draw (-0.4,3.5) node {$N_2(x)$};
\draw (1.4,3.5) node {$N_2(y)$};
\draw (1.5,4.8) node {$V^*(G)$};
\end{tikzpicture}
\end{center}
\caption{Construction from the proof of Lemma~\ref{conn-r-ASC}}
\label{fig:at-most-2r-1}
\end{figure}

Clearly, $w$ is pendant in $H$ with $wy_1\in E(H)$. Moreover, $d_H(w,x_{r-2})=r+1$. We have ${\rm ecc}_H(v)=r$ for any vertex $v\in V_0\cup \{x,y\}$ with $v\neq x_{r-2}$ from the fact that $V_0\cup \{x,y\}$ induces an isometric subgraph $C_{2r}'$ of $H$. Similarly we have ${\rm ecc}_H(u)=r$ for any vertex $u\in V(G)\setminus \{x,y\}$. Therefore $H$ is an $r$-ASC embedding graph of  $G$, finishing the proof of this lemma.
\end{proof}

We are now ready for the main result of this section which further sharpens Theorem~\ref{thm:constuction} by proving that $\theta_r(G) \le 2r-1$ holds if $r\ge 3$ and $G$ has at least two vertices.

\begin{theorem}\label{index-2r}
If $G$ is a graph of order $n$ and $r\geq 3$, then $\theta_r(G)=2r$ if and only if $G\cong K_1$.
\end{theorem}

\begin{proof} Note that $C_{2r}'$ is an $r$-ASC graph.
If $G\cong K_1$, considering the fact that any $r$-ASC has order $n\geq 2r+1$ \cite{KNW2011}, the result $\theta_r(G)=2r$ holds immediately. Conversely, if $G$ is a graph of order $n\geq 2$, in view of Theorem \ref{thm:constuction}, we only need to prove that $\theta_r(G)\leq 2r-1$. Equivalently, it suffices to construct a graph $H$ of order at most $n+2r-1$ as an $r$-ASC embedding graph of $G$. Denote by $h$ the number of isolated vertices in $G$ and set $V_0=\{x_1,\ldots,x_{r-1},y_1,\ldots,y_{r-1},w\}$. Based on the value of $h$, we distinguish the following three cases.

\medskip\noindent
{\bf Case 1.} $h\geq 2$.\\
In this case we may assume that $x_0,y_0$ are two isolated vertices in $G$. In $H$, the adjacency relation of $V_0\cup \{x_0,y_0\}$ is $x_0-x_1-\cdots-x_{r-1}-y_0-y_{r-1}-\cdots-y_1-x_0$ with $x_0w$ being a pendant edge, that is, $V_0\cup \{x_0,y_0\}$ induces a graph $C_{2r}'$ in $H$. Moreover, any other vertex $v$ than $x_0,y_0$ is joined with $x_1$ and $y_1$ in $H$. The construction is shown in Fig.~\ref{fig:Case1-for-iff-K1}, where the thick edges from $x_1$ and $y_1$ indicate that $x_1$ and $x_2$ are adjacent to all the vertices in $V(G)\setminus \{x_0,y_0\}$.

\begin{figure}[ht!]
\begin{center}
\begin{tikzpicture}[scale=1.0,style=thick]
\def\vr{3pt}
\path (0,0) coordinate (x0);
\path (5,0) coordinate (y0);
\path (-1,1) coordinate (w);
\path (1,2) coordinate (x1);
\path (2,2) coordinate (x2);
\path (4,2) coordinate (xr-1);
\path (6,2) coordinate (y1);
\path (7,2) coordinate (y2);
\path (9,2) coordinate (yr-1);
\draw (x0) -- (w);
\draw (y0) -- (xr-1);
\draw (x0) -- (y1) --(y2);
\draw (y0) -- (yr-1);
\draw (y0) -- (xr-1);
\draw (x0) --  (x1) -- (x2);
\draw[line width=1mm] (x1) -- (1.5,-0.4);
\draw[line width=1mm] (y1) -- (6.5,-0.4);
\draw (x2) -- (2.5,2);
\draw (xr-1) -- (3.5,2);
\draw (y2) -- (7.5,2);
\draw (yr-1) -- (8.5,2);
\draw (x0)  [fill=white] circle (\vr);
\draw (y0)  [fill=white] circle (\vr);
\draw (w)  [fill=white] circle (\vr);
\draw (x1)  [fill=white] circle (\vr);
\draw (y1)  [fill=white] circle (\vr);
\draw (x2)  [fill=white] circle (\vr);
\draw (y2)  [fill=white] circle (\vr);
\draw (xr-1)  [fill=white] circle (\vr);
\draw (yr-1)  [fill=white] circle (\vr);
\draw (3,0) ellipse (150pt and 30pt);
\draw[below] (x0)++(0,-0.1) node {$x_0$};
\draw[below] (y0)++(0,-0.1) node {$y_0$};
\draw[above] (w)++(0,0.1) node {$w$};
\draw[above] (x1)++(0,0.1) node {$x_1$};
\draw[above] (y1)++(0,0.1) node {$y_1$};
\draw[above] (x2)++(0,0.1) node {$x_2$};
\draw[above] (y2)++(0,0.1) node {$y_2$};
\draw[above] (xr-1)++(0,0.1) node {$x_{r-1}$};
\draw[above] (yr-1)++(0,0.1) node {$y_{r-1}$};
\draw (3,-0.5) node {$G$};
\draw (3,2) node {$\cdots$};
\draw (8,2) node {$\cdots$};
\draw (-2,2) node {$H:$};
\end{tikzpicture}
\end{center}
\caption{Case 1 from the proof of Theorem~\ref{index-2r}}
\label{fig:Case1-for-iff-K1}
\end{figure}

It can be easily checked that $d_H(w,y_0)=r+1$ and any other vertex in $H$ has eccentricity $r$. This ensures that $H$ is an $r$-ASC embedding graph of $G$.

\medskip\noindent
{\bf Case 2.} $h=1$.\\
In this case let $v$ be an isolated vertex in $G$.  Let the adjacency relation of $V_0\setminus\{w\}$ and $v$ in $H$ be $x_1-\cdots-x_{r-1}-v-y_{r-1}-\cdots-y_1$. Any  vertex of $G$ different from $v$ is joined with $x_1$ and $y_1$ in $H$. Choosing an arbitrary but fixed vertex $u\in V(G)\setminus\{v\}$, we join $w$ with $u$. This defines $H$.  See Fig.~\ref{fig:Case2-for-iff-K1} where this construction is shown; again thick edges from $x_1$ and $y_1$ indicate that $x_1$ and $y_1$ are adjacent to all the vertices in $V(G)\setminus \{v\}$.

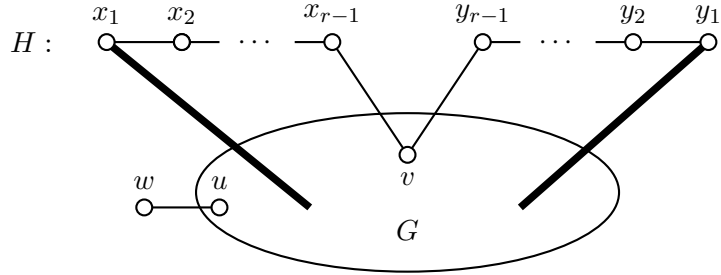
\begin{figure}[ht!]
\begin{center}
\begin{tikzpicture}[scale=1.0,style=thick]
\def\vr{3pt}
\path (0,0.5) coordinate (v);
\path (-1,2) coordinate (xr-1);
\path (-3,2) coordinate (x2);
\path (-4,2) coordinate (x1);
\path (1,2) coordinate (yr-1);
\path (3,2) coordinate (y2);
\path (4,2) coordinate (y1);
\path (-2.5,-0.2) coordinate (u);
\path (-3.5,-0.2) coordinate (w);
\draw (xr-1) -- (v) -- (yr-1);
\draw (y1) -- (y2);
\draw (x1) -- (x2);
\draw (xr-1) -- (-1.5,2);
\draw (yr-1) -- (1.5,2);
\draw (x2) -- (-2.5,2);
\draw (y2) -- (2.5,2);
\draw (u) -- (w);
\draw[line width=1mm] (x1) -- (-1.3,-0.2);
\draw[line width=1mm] (y1) -- (+1.5,-0.2);
\draw (x1)  [fill=white] circle (\vr);
\draw (x2)  [fill=white] circle (\vr);
\draw (w)  [fill=white] circle (\vr);
\draw (y1)  [fill=white] circle (\vr);
\draw (y2)  [fill=white] circle (\vr);
\draw (x2)  [fill=white] circle (\vr);
\draw (u)  [fill=white] circle (\vr);
\draw (v)  [fill=white] circle (\vr);
\draw (xr-1)  [fill=white] circle (\vr);
\draw (yr-1)  [fill=white] circle (\vr);
\draw (0,0) ellipse (80pt and 30pt);
\draw[below] (v)++(0,-0.1) node {$v$};
\draw[above] (w)++(0,0.1) node {$w$};
\draw[above] (u)++(0,0.1) node {$u$};
\draw[above] (x1)++(0,0.1) node {$x_1$};
\draw[above] (y1)++(0,0.1) node {$y_1$};
\draw[above] (x2)++(0,0.1) node {$x_2$};
\draw[above] (y2)++(0,0.1) node {$y_2$};
\draw[above] (xr-1)++(0,0.1) node {$x_{r-1}$};
\draw[above] (yr-1)++(0,0.1) node {$y_{r-1}$};
\draw (0,-0.5) node {$G$};
\draw (2,2) node {$\cdots$};
\draw (-2,2) node {$\cdots$};
\draw (-5,2) node {$H:$};
\end{tikzpicture}
\end{center}
\caption{Case 2 from the proof of Theorem~\ref{index-2r}}
\label{fig:Case2-for-iff-K1}
\end{figure}

Note that $d_H(w,v)=r+1$ with $w$ being pendant and any other vertex than $w,v$ has eccentricity $r$ in $H$. Thus $H$ is an $r$-ASC embedding graph of $G$ as desired.

\medskip\noindent
{\bf Case 3.} $h=0$. \\
 In this case, if $G$ is a connected graph, then we are done by Lemma \ref{conn-r-ASC}. If not, by choosing any component of $G$,  we run the same operation on each component of the graph $G$ as that in the proof of Lemma \ref{conn-r-ASC} and join any vertex $v$ from other component(s) with $y_1,y_3$ for $r\geq 4$ or with $y_1,x_2$ for $r=3$. In this way we obtain an $r$-ASC embedding graph of $G$, completing the proof of the theorem.
\end{proof}

\section{$3$-ASC index of graphs with diameter at most $2$}
\label{sec:small-diam}

In this section we first determine the $3$-ASC index of graphs of diameter $1$.  Clearly, a graph of order $n>1$ has diameter $1$ if and only if $G\cong K_n$. Next we prove the $3$-ASC index of graphs with diameter $1$. For consistency we include $K_1$ in the following theorem.

\begin{theorem}\label{thm-complete}
$\theta_3(K_1)=6$ and if $n\ge 2$, then $\theta_3(K_n)=5$.
\end{theorem}

\begin{proof}
By inspection on small graphs we infer that the order of a smallest $3$-ASC graph is $7$. Consequently, $\theta_3(K_1)=6$. In the rest of the proof we assume that $n\ge 2$.

We first show that $\theta_3(K_n)\geq 5$. Assume that $G$ is a graph obtained from $K_n$ by adding at most three vertices. Then ${\rm diam}(G)\leq 4$. Moreover, the equality holds if and only if either $G\cong G^{\prime}$ is formed by joining a pendant vertex of a path $P_3$ to some but not all vertices of $K_n$, or $G\cong G^{\prime\prime}$ is obtained by joining a vertex of a path $P_2$ to some but not all vertices of $K_n$ and joining another new vertex to some other different vertices with degree $n-1$ in it. In both cases there is at least one vertex $x$ in $G$ with ${\rm ecc}_G(x)=2$. Thus $G$ is not a $3$-ASC graph which in turn implies that $\theta_3(K_n)\geq 4$.

Suppose next that $\theta_3(K_n)=4$ and let $G$ be a $3$-ASC graph obtained by adding vertices $V_0=\{x,y,z,w\}$ to $K_n$. If $G[V_0]\ncong P_4$, then $G$ contains $G^{\prime}$ or $G^{\prime\prime}$ as an induced graph, where $G^{\prime}$ and $G^{\prime\prime}$ are the graphs defined in the above paragraph. Otherwise, ${\rm diam}(G)<4$, a clear contradiction. Since ${\rm diam}(G) = 4$, there exists at least one vertex $v$ in $G$ with ${\rm ecc}_G(v)=2$, but this contradicts the fact that $G$ is a $3$-ASC graph.

 Hence $G[V_0]\cong P_4$ must hold. Let the adjacency relation of $V_0$ in $G$ be $w-z-y-x$ and assume without loss of generality that $w$ has a neighbor in $K_n$. Then $x$ is a pendant vertex of $G$ because otherwise any vertex of $G$ would have eccentricity at most $3$. If at least one  of the vertices $y$ and $z$, say $y$, is adjacent to one or more vertices of $K_n$, then ${\rm ecc}_G(y)=2$, which is a contradiction to the fact that $G$ is a $3$-ASC graph. Therefore, ${\rm deg}_G(z) = {\rm deg}_G(y) = 2$. Now, if $w$ is adjacent to all the vertices of $K_n$, then ${\rm ecc}_G(z)=2$, otherwise ${\rm ecc}_G(x)=5$. Hence in  both cases $G$ is not a $3$-ASC graph. We conclude that $\theta_3(K_n)\geq 5$.

\begin{figure}[ht!]
\begin{center}
\begin{tikzpicture}[scale=0.8,style=thick]
\def\vr{3pt}
\path (0,0) coordinate (a);
\path (2,0) coordinate (b);
\path (4,0) coordinate (c);
\path (1,2) coordinate (d);
\path (3,2) coordinate (e);
\path (5,2) coordinate (f);
\path (2,4) coordinate (g);
\path (4,4) coordinate (h);
\draw (a) -- (b) -- (c) -- (f) -- (h) -- (g) --(d) -- (b);
\draw (b) -- (d) -- (g);
\draw (c) -- (e) -- (h);
\draw (a)  [fill=white] circle (\vr);
\draw (b)  [fill=black] circle (\vr);
\draw (c)  [fill=black] circle (\vr);
\draw (d)  [fill=black] circle (\vr);
\draw (e)  [fill=black] circle (\vr);
\draw (f)  [fill=black] circle (\vr);
\draw (g)  [fill=black] circle (\vr);
\draw (h)  [fill=white] circle (\vr);
\draw (4.5,2) ellipse (55pt and 15pt);
\draw (4,0.7) node {$\cdots$};
\draw (4,2) node {$\cdots$};
\draw (4,3.3) node {$\cdots$};
\draw [right] (f) node {$K_{n-1}$};
\draw [right] (h) node {$u$};
\draw [right] (c) node {$v$};
\end{tikzpicture}
\end{center}
\caption{Graph $G_n$, $n\geq 1$}
\label{fig:Q1}
\end{figure}
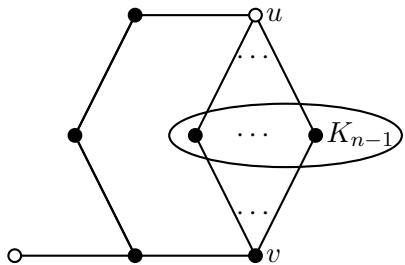

It remains to prove that $\theta_3(K_n)\le 5$ for $n\ge 2$. For this sake consider the graph $G_n$ which is schematically shown in Fig.~\ref{fig:Q1}. In the graph $G_n$, the vertices $u$ and $v$ are adjacent to all the vertices of the $K_{n-1}$ leading to two subgraphs isomorphic to $K_n$.  The two white vertices of $G_n$ have eccentricity $4$ while all the black vertices have eccentricity $3$  (hereafter the white vertices denote the diametrical vertices and black vertices denote central ones). Hence $G_n$ is a $3$-ASC graph that contains $K_n$.
\end{proof}

 The graph $G_2$ from the proof of Theorem~\ref{thm-complete} is of order seven which makes it a smallest $3$-ASC graph. Now we turn to  graphs of diameter $2$ and first bound their $3$-ASC index as follows.

\begin{lemma}\label{lem:3-or-4}
 If $G$ is a graph of diameter $2$, then $3\leq\theta_3(G)\leq 4$.
\end{lemma}

\begin{proof}
To prove the upper bound we need to construct a graph $H$ of order $|V(G)|+4$ as a $3$-ASC embedding graph of $G$. Set $V(H)=V(G)\cup\{w,x,y,z\}$ and let the adjacency relation between the new vertices be $w-z-y-x$. Let $v$ be a vertex in $G$ with ${\rm ecc}_G(v)=2$ and $u$ an eccentric vertex of $v$ in $G$, that is, $d_G(v,u)=2$. Join the vertex $u$ in $G$ with $x$ and join all vertices from $V(G)\setminus \{u,v\}$ with $w$.  This defines $H$ (see Fig.~\ref{fig:QNEW1} for an example of $H$ when $G=K_{2,3}$).

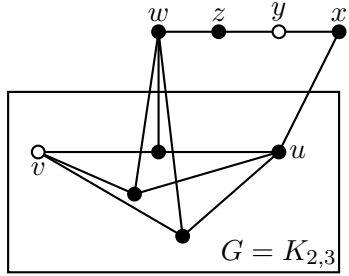
\begin{figure}[ht!]
\begin{center}
\begin{tikzpicture}[scale=0.8,style=thick]
\def\vr{3pt}
\path (0,0.5) coordinate (a);
\path (2,0.5) coordinate (b);
\path (4,0.5) coordinate (c);
\path (1.6,-0.2) coordinate (d);
\path (2,2.5) coordinate (e);
\path (3,2.5) coordinate (f);
\path (4,2.5) coordinate (g);
\path (5,2.5) coordinate (h);
\path (2.4,-0.9) coordinate (i);
\draw (d) -- (e);
\draw (c) -- (i) -- (a);
\draw (i) -- (e) -- (b);
\draw (a) -- (d) -- (c);
\draw (a) -- (b) -- (c);
\draw (e) -- (f) -- (g) -- (h);
\draw (c) -- (h);
\draw (-0.5,-1.5) rectangle (5,1.5);
\draw (4,-1.2) node {$G=K_{2,3}$};
\draw (a)  [fill=white] circle (\vr);
\draw (b)  [fill=black] circle (\vr);
\draw (c)  [fill=black] circle (\vr);
\draw (d)  [fill=black] circle (\vr);
\draw (e)  [fill=black] circle (\vr);
\draw (f)  [fill=black] circle (\vr);
\draw (g)  [fill=white] circle (\vr);
\draw (h)  [fill=black] circle (\vr);
\draw (i)  [fill=black] circle (\vr);
\draw [below] (a) node {$v$};
\draw [right] (c) node {$u$};
\draw [above] (h) node {$x$};
\draw [above] (g) node {$y$};
\draw [above] (f) node {$z$};
\draw [above] (e) node {$w$};
\end{tikzpicture}
\end{center}
\caption{ Graph $H$ for $G=K_{2,3}$}
\label{fig:QNEW1}
\end{figure}

Note that ${\rm deg}_H(y)={\rm deg}_H(z)={\rm deg}_H(x)=2$ and that the neighborhood of $v$ remains unchanged in $H$.  Then we arrive at $d_H(v,y)=4$ from the structure of $H$. Moreover, any vertex $v^{\prime}$ from $N_G(u)$ has eccentricity $3$ in $H$ with ${\rm Ecc}_H(v^{\prime})=\{y\}$. Similarly, ${\rm Ecc}_H(v^{\prime})=\{x,y\}$ for any vertex $v^{\prime}\in {\rm Ecc}_G(u)\setminus\{v\}$, since both $x$ and $y$ have a largest distance $3$ to $v^{\prime}$ in $H$. It follows that ${\rm ecc}_H(v^{\prime})=3$ for any vertex $v^{\prime}\in {\rm Ecc}_G(u)\setminus\{v\}$. In addition, we have $d_H(w,x)=3=d_H(z,u)$, which implies that ${\rm ecc}_H(u)={\rm ecc}_H(z)=3={\rm ecc}_H(w)={\rm ecc}_H(x)$. Thus $H$ is a $3$-ASC embedding graph of $G$ as desired.

Next we deal with the lower bound on $\theta_3(G)$. Note that any graph obtained by adding one more vertex to $G$ has diameter at most $3$. Thus we have $\theta_3(G)\geq 2$.

Assume that $G^{\prime}$ is an embedding
$3$-ASC graph of  $G$ by adding two more vertices $x$ and $y$. Then $G^{\prime}$ has diameter $4$. Hence the adjacency relation of
$x,y$ in $G^{\prime}$ is either $y-x-z$ with $y$ being pendant, ${\rm deg}_{G^{\prime}}(x)\geq 2$ and $z\in V(G)$ with ${\rm ecc}_G(z)=2$, or $x-u$ and $y-v$ where $d_G(u,v)=2$ with ${\rm ecc}_G(u)={\rm ecc}_G(v)=2$. Considering that $G$ has diameter $2$, in the latter case, there is at least one vertex $w$ as a common neighbor of vertices $u$ and $v$ in $G$. Then we deduce that ${\rm ecc}_{G^{\prime}}(z)=2$ in the former case or ${\rm ecc}_{G^{\prime}}(w)=2$ in the latter case, either of which contradicts the fact that $G^{\prime}$ is a $3$-ASC graph. So $\theta_3(G)\geq 3$, finishing the proof of this lemma.
\end{proof}

It is interesting to note that if $G=P_3$, then the graph $H$ constructed in the proof of Lemma~\ref{lem:3-or-4} is the graph of order $7$ mentioned before the lemma. Since the order of a smallest $3$-ASC graph is $7$, it follows that $\theta_3(P_3) = 4$. There is a unique $2$-SC graph $C_4$ of order $4$ with diameter $2$. By Lemma~\ref{lem:3-or-4}, $\theta_3(C_4)\leq 4$ and one can prove that actually $\theta_3(C_4) = 4$ holds. Consider next $C_5$. By Lemma~\ref{lem:3-or-4}, $\theta_3(C_5)\leq 4$. But, to our surprise, $4$ is not the minimum number of vertices needed to be added for obtaining a $3$-ASC embedding of $C_5$. Denote by $C_n'$ the graph which consists of the cycle $C_n$ and an additional vertex adjacent to exactly two  consecutive vertices in $C_n$.  As observed in~\cite{KNW2011}, the graph $C_{2r+1}'$ is an $r$-ASC graph. Let $H_5'$  (shown in Fig.~\ref{fig:QNEW2}) be the graph obtained from $C_5'$ by respectively attaching a pendant vertex to the two non-adjacent vertices of degree $2$ in $C_5'$. Since $H_5'$ is a $3$-ASC graph, we conclude that $\theta_3(C_5)=3$. This example can be generalized as follows.

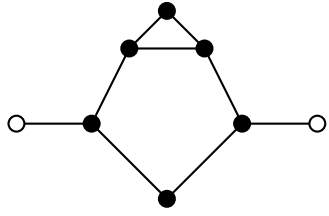
\begin{figure}
\begin{center}
\begin{tikzpicture}[scale=1,style=thick]
\def\vr{3pt}

\draw (1.5,2) -- (2,1) -- (1,0) -- (0,1) -- (0.5,2);
\draw (1,2.5) -- (0.5,2) -- (1.5,2) -- (1,2.5);
\draw (2,1) -- (3,1);
\draw (-1,1) -- (0,1);

\draw (0,1)  [fill=black] circle (\vr);
\draw (-1,1)  [fill=white] circle (\vr);
\draw (2,1)  [fill=black] circle (\vr);
\draw (3,1)  [fill=white] circle (\vr);
\draw (1,0)  [fill=black] circle (\vr);
\draw (0.5,2)  [fill=black] circle (\vr);
\draw (1.5,2)  [fill=black] circle (\vr);
\draw (1,2.5)  [fill=black] circle (\vr);
\end{tikzpicture}
\end{center}
\caption{ Graph $H_5'$} \label{fig:QNEW2}
\end{figure}
\begin{theorem}\label{thm:new-added}
Let $G$ be a $2$-SC graph. If $G$ contains a diametrical pair $u,v$ and vertices $u',v'$ such that $u'\in N_G(u)\setminus N_G(v)$, $v'\in N_G(v)\setminus N_G(u)$ and $N_G(u)\cap N_G(v)\subseteq {\rm Ecc}_G(u')\cap {\rm Ecc}_G(v')$, then $\theta_3(G)=3$.
\end{theorem}

\begin{proof} By Lemma~\ref{lem:3-or-4} it suffices to construct a $3$-ASC embedding graph of $G$. Let $H$ be the graph obtained from $G$ by adding three new vertices $x,y,z$ and edges $xu$, $yv$, $zu'$, and $zv'$. The vertices $x$ and $y$ are pendant in $H$ and hence $d_H(x,y)=4$. Consequently, ${\rm ecc}_H(u)={\rm ecc}_H(v)=3$ with ${\rm Ecc}_H(u)=\{y\}$ and ${\rm Ecc}_H(v)=\{x\}$. Since $N_G(u)\cap N_G(v)\subseteq Ecc_G(u')\cap Ecc_G(v')$, any vertex $w\in N_G(u)\cap N_G(v)$ is at distance $2$ from both $u'$ and $v'$ in $G$. Thus any vertex $w\in N_G(u)\cap N_G(v)$ is an eccentric vertex of $z$ in $H$ with $d_H(w,z)=3$. Then ${\rm ecc}_H(z)=3$ holds. This in turn implies that ${\rm ecc}_H(w)=3$ for any vertex $w\in N_G(u)\cap N_G(v)$. Moreover, all the vertices in $N_G(u)\setminus N_G(v)$, including $u'$, have eccentricity $3$ in $H$ with $y$ as their common eccentric vertex. By symmetry, all the vertices in $N_G(v)\setminus N_G(u)$, including $v'$, have eccentricity $3$ in $H$. Finally, each vertex (if any) in ${\rm Ecc}_G(u)\cap {\rm Ecc}_G(v)$ has eccentricity $3$ in $H$ with $x$ and $y$ as its eccentric vertices. We conclude that $H$ is a $3$-ASC embedding graph of $G$.
\end{proof}

By Lemma~\ref{lem:3-or-4}, the $3$-ASC index of a graph $G$ of diameter $2$ is either $3$ or $4$. In the rest of this section we make a partial corresponding classification, where we  will always assume that the graph $G$ considered is of order $n$ and of diameter $2$. We first prove:

\begin{theorem}
\label{thm:no-triple}
If $G$ does not contain a diametrical triple and does not satisfy the assumption of Theorem~\ref{thm:new-added}, then $\theta_3(G) = 4$.
\end{theorem}

\begin{proof}
By Lemma~\ref{lem:3-or-4} it suffices to prove that $\theta_3(G)\neq 3$.  Suppose on the contrary that $G^{*}$ is a $3$-ASC embedding graph of the graph $G$ with $V(G^*)=V(G)\cup \{x,y,z\}$ and set $V_0=\{x,y,z\}$.

 Assume first that $G$ has a vertex $v$  of degree $n-1$. Note that then there exists  a vertex $u$ with ${\rm ecc}_G(u)=2$ adjacent to $v$ in $G$.  We now distinguish two cases.

\medskip\noindent
{\bf Case 1.} ${\rm ecc}_{G^{*}}(v)=4$.\\
In this case the adjacency relation of $V_0$ is $x-y-z$ with one vertex, say $z$, being pendant. Note that $G^*$ has diameter $4$. We  infer that any vertex $u$ with ${\rm ecc}_G(u)=2$ must be adjacent to $x$ in $G^{*}$.   But then ${\rm ecc}_{G^{*}}(x)=2$, which is impossible because of the $3$-ASC property of $G^{*}$.

\medskip\noindent
{\bf Case 2.} ${\rm ecc}_{G^{*}}(v)=3$. \\
In this case, clearly, we have ${\rm Ecc}_{G^*}(v)\subseteq V_0$. Since ${\rm ecc}_{G^*}(v)={\rm ecc}_G(v)+2$, we have $|{\rm Ecc}_{G^*}(v)|\leq 2$. If $|{\rm Ecc}_{G^*}(v)|=2$, we may assume that ${\rm Ecc}_{G^*}(v)=\{y,z\}$. Then $x\in N_{G^*}(y)\cap N_{G^*}(z)$ and there exists a vertex $w$ with $wx\in E(G^*)$. But now we get ${\rm ecc}_{G^*}(w)=2$, which is a contradiction. When $|{\rm Ecc}_{G^*}(v)|=1$, we may assume that ${\rm Ecc}_{G^*}(v)=\{z\}$ and $v-w-y-z$ is a $(v,z)$-path of length $3$. Note that ${\rm ecc}_{G^{*}}(w)=3$. Now we consider the position of $x$ in $G^*$. If $xv\in E(G^*)$, then ${\rm ecc}_{G^{*}}(w)=2$, a contradiction again. If $d_{G^*}(x,v)=2$, we assume that $w^{\prime}\in N_{G^*}(x)\cap N_{G^*}(v)$.  If $ww^{\prime}\in E(G)$, we get the same result ${\rm ecc}_{G^{*}}(w)=2$. We observe that $x$ is not adjacent to any vertex from $N_{G^*}(w)$. If not, we have ${\rm ecc}_{G^{*}}(w)=2$ which is impossible. In particular, $xy\notin E(G^*)$. Now we only need to deal with the subcase when $d_G(w,w^{\prime})=2$. Let $w^{\prime\prime}$ be a common neighbor of $w$ and $w^{\prime}$ in $G$. If $xz\in E(G^*)$, then $xw^{\prime}vwyzx$ is an induced cycle $C_6$. Considering that $G$ contains no diametrical triple, then $G^*$ is a $3$-SC graph. This contradicts to the $3$-ASC property of $G^*$. Therefore $xz\notin E(G^*)$. Combining $d_{G^*}(v,z)=3$ with ${\rm deg}_G(v)=n-1$, we find that $z$ is pendant in $G^*$. Obviously, ${\rm ecc}_{G^*}(z)={\rm ecc}_{G^*}(y)+1$. It follows that ${\rm ecc}_{G^*}(z)=4$ from the $3$-ASC property of $G^*$. Then $2\leq d_{G^*}(y,x)\leq 3$. If $d_{G^*}(y,x)=3$, then there must be a $(y,x)$-path, say $y-y_1-x_1-x$, of length $3$ where $y_1,x_1\in V(G)\setminus\{v\}$.  This deduces ${\rm ecc}_{G^*}(y_1)=2$, a contradiction again. While if $d_{G^*}(y,x)=2$, there must be a neighbor $v^{\prime}$ of $v$ such that $v^{\prime}$ is a common neighbor of $x$ and $y$ in $G^*$. Then ${\rm ecc}_{G^*}(v^{\prime})=2$.  Both these possibilities are impossible from the $3$-ASC property of $G^*$.

\medskip
We have thus proved that no vertex of $G$ can be of degree $n-1$, that is, $G$ must be a $2$-SC graph. If the adjacency relation of $V_0$ in $G^*$ is $x-y-z$, then there is one vertex, say $z$, in $V_0$ being pendant in $G^{*}$. Otherwise, each vertex from  $V_0$ has degree at least $2$ in $G^*$. If ${\deg}_{G^*}(y)>2$, then any vertex in $G^*$ has eccentricity at most $3$. Thus $G^*$ can not be a $3$-ASC graph. Therefore ${\deg}_{G^*}(y)=2$, i.e., none of the vertices in $G$ is adjacent to $y$ in $G^*$. Assume that $xx^{\prime},zz^{\prime}\in E(G^*)$ with $x^{\prime},z^{\prime}\in V(G)$. If  $d_G(x^{\prime},z^{\prime})=2$, then we also claim that any vertex in $G^{*}$ has eccentricity at most $3$ since ${\rm Ecc}_G(u)\cap {\rm Ecc}_G(v)=\emptyset$ for any diametrical pair $u,v$ in $G$ from the condition in this theorem. While if $d_G(x^{\prime},z^{\prime})=1$ or $x^{\prime}=z^{\prime}$, then $G^*$ has a vertex with eccentricity $2$ in $G^*$. All of them are clear contradictions to the $3$-ASC property. Assume that $vx\in E(G^*)$ where $v\in V(G)$. Considering that $G^*$ has diameter $4$, we conclude that any vertex $u$ in ${\rm Ecc}_G(v)$ must be adjacent to $x$ or $y$ in $G^*$. But now ${\rm ecc}_{G^*}(x)=2$, which is impossible.

If $G^*[V_0]=P_2\cup K_1$, without loss of generality, assume that $P_2=xy$ with the adjacency relation $v-x-y$  in $G^*$ where $v$ is an arbitrary vertex in $G$. Considering that $G^*$ has diameter $4$, we have ${\rm ecc}_{G^*}(v)=2$, which is still a contradiction. While $G^*[V_0]=\overline{K_3}$, considering that $G$ does not satisfy the assumption of Theorem~\ref{thm:new-added}, we can find a vertex $w\in V(G)$ such that ${\rm ecc}_{G^*}(w)=2$. A contradiction occurs again, which completes the proof of this theorem.
\end{proof}

By Theorem~\ref{thm:no-triple} it remains to consider those graphs $G$ of diameter $2$ that contain diametrical triples. We next give two sufficient conditions that guarantee that $\theta_3(G) =  3$.

\begin{theorem}\label{thm:isolated-vertex}
 If a graph $G$ contains vertices $u$ and $v$ of a diametrical triple of $G$ such that $\delta\left(G[{\rm Ecc}_G(u)\cap {\rm Ecc}_G(v)]\right) = 0$, then $\theta_3(G) =  3$.
\end{theorem}

\begin{proof}
By Lemma~\ref{lem:3-or-4} we only need to prove that $\theta_3(G) \le 3$.

Since $u$ and $v$ belong to a diametrical triple, there exists a vertex $w\in {\rm Ecc}_G(u)\cap{\rm Ecc}_G(v)$, $w\ne u,v$. By the theorem's assumption we may assume that $w$ is an isolated vertex in $G[{\rm Ecc}_G(u)\cap {\rm Ecc}_G(v)]$. Now we construct a graph $H$ as follows. Let $V(H) = V(G) \cup \{x,y,z\}$ and $E(G)\cup \{ux,vz,xy,yz\}\subseteq E(H)$. In addition, we join $x$ to all (if any) vertices from ${\rm Ecc}_G(u)\cap{\rm Ecc}_G(v)$ different from $w$.  This construction is presented in Fig.~\ref{fig:diam-triples}.

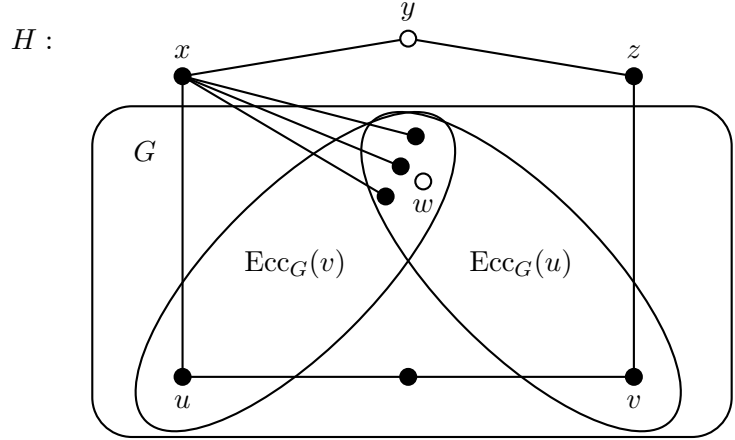
\begin{figure}[ht!]
\begin{center}
\begin{tikzpicture}[scale=1.0,style=thick]
\def\vr{3pt}
\path (0,0) coordinate (xxx);
\path (-3,0) coordinate (u);
\path (3,0) coordinate (v);
\path (-3,4) coordinate (x);
\path (0,4.5) coordinate (y);
\path (3,4) coordinate (z);
\path (0.2,2.6) coordinate (w);
\path (-0.3,2.4) coordinate (w1);
\path (-0.1,2.8) coordinate (w2);
\path (0.1,3.2) coordinate (w3);
\draw (u) -- (xxx) -- (v);
\draw (x) -- (y) -- (z);
\draw (x) -- (w1);
\draw (x) -- (w2);
\draw (x) -- (w3);
\draw (x) -- (u);
\draw (z) -- (v);
\draw (y)  [fill=white] circle (\vr);
\draw (w)  [fill=white] circle (\vr);
\draw (x)  [fill=black] circle (\vr);
\draw (z)  [fill=black] circle (\vr);
\draw (u)  [fill=black] circle (\vr);
\draw (xxx)  [fill=black] circle (\vr);
\draw (v)  [fill=black] circle (\vr);
\draw (w1)  [fill=black] circle (\vr);
\draw (w2)  [fill=black] circle (\vr);
\draw (w3)  [fill=black] circle (\vr);
\draw[rotate around={45:(-1.5,1.4)}] (-1.5,1.4) ellipse (80pt and 30pt);
\draw[rotate around={-45:(1.5,1.4)}] (1.5,1.4) ellipse (80pt and 30pt);
\draw[below] (u)++(0,-0.1) node {$u$};
\draw[below] (v)++(0,-0.1) node {$v$};
\draw[above] (x)++(0,0.1) node {$x$};
\draw[above] (y)++(0,0.1) node {$y$};
\draw[above] (z)++(0,0.1) node {$z$};
\draw[below] (w)++(0,-0.1) node {$w$};
\draw[rounded corners=15pt] (-4.2,-0.8) rectangle (4.3,3.6);
\draw (-3.5,3) node {$G$};
\draw (-5,4.5) node {$H:$};
\draw (-1.5,1.5) node {${\rm Ecc}_G(v)$};
\draw (1.5,1.5) node {${\rm Ecc}_G(u)$};
\end{tikzpicture}
\end{center}
\caption{Construction from the proof of Theorem~\ref{thm:isolated-vertex}}
\label{fig:diam-triples}
\end{figure}

Since $w$ is an isolated vertex in $G[{\rm Ecc}_G(u)\cap {\rm Ecc}_G(v)]$, we have $d_G(w,w^{\prime})=2$ for any vertex $w^{\prime}\in \left({\rm Ecc}_G(u)\cap{\rm Ecc}_G(v)\right)\setminus\{w\}$. Clearly, $d_G(w,u)=2=d_G(w,v)$. So $d_H(w,y)=4$. Any vertex from $N_G(u)\cap N_G(v)$ has eccentricity $3$ since $y$ is its unique eccentric vertex with distance $3$ to it in $H$. In addition, ${\rm ecc}_H(u)=3$ from the fact that ${\rm Ecc}_H(u)=\{z\}$ with $d_H(u,z)=3$. Similarly, we also have ${\rm ecc}_H(v)=3$. For any vertex $w^{\prime}\in {\rm Ecc}_G(u)\cap{\rm Ecc}_G(v)\setminus\{w\}$, considering that $w^{\prime}x\in E(H)$, we conclude that ${\rm ecc}_H(w^{\prime})=3$ since $z$ is the only one vertex with the largest distance $3$ to $w^{\prime}$ in $H$. For any vertex $v^{\prime}\in N_G(u)\setminus N_G(v)$, we have ${\rm Ecc}_H(v^{\prime})=\{y,z\}$ with $d_H(v^{\prime},y)=3=d_H(v^{\prime},z)$ being largest. Then it follows that ${\rm ecc}_H(v^{\prime})=3$. Moreover, ${\rm ecc}_H(v^{\prime})=3$ with ${\rm Ecc}_H(v^{\prime})=\{y,x\}$ for any vertex $v^{\prime}\in N_G(v)\setminus N_G(u)$ by an analogous reasoning.  We conclude that $H$ is a $3$-ASC embedding graph of $G$.
\end{proof}

\begin{theorem}
\label{thm:P3}
 If a graph $G$ contains vertices $u$ and $v$ of a diametrical triple of $G$ such that $G[{\rm Ecc}_G(u)\cap {\rm Ecc}_G(v)]$ contains $P_3$ as an induced subgraph, then $\theta_3(G) =  3$.
\end{theorem}

\begin{proof}
Applying Lemma~\ref{lem:3-or-4} again we only need to prove that $\theta_3(G) \le 3$.

Let $w_1w_2w_3$ be an induced path in $G[{\rm Ecc}_G(u)\cap {\rm Ecc}_G(v)]$. We construct a graph $H$ obtained from $G$ by adding vertices $x,y,z$ and edges $xy, yz, ux, vz, w_3x$. Moreover, we join all vertices (if any) from $\left({\rm Ecc}_G(u)\cap {\rm Ecc}_G(v)\right)\setminus N_G[w_1]$ with the vertex $x$ in $H$.  See Fig.~\ref{fig:diam-triples-P3} for an illustration of the construction.

\begin{figure}[ht!]
\begin{center}
\begin{tikzpicture}[scale=1.2,style=thick]
\def\vr{3pt}
\path (0,0) coordinate (xxx);
\path (-3,0) coordinate (u);
\path (3,0) coordinate (v);
\path (-3,4) coordinate (x);
\path (0,4.5) coordinate (y);
\path (3,4) coordinate (z);
\path (-0.7,2.5) coordinate (w3);
\path (0,2.0) coordinate (w2);
\path (0.5,2.5) coordinate (w1);
\draw (u) -- (xxx) -- (v);
\draw (x) -- (y) -- (z);
\draw (x) -- (w3);
\draw (x) -- (u);
\draw (z) -- (v);
\draw (w1) -- (w2) -- (w3);
\draw[line width=1mm] (x) -- (-0.3,2.8);
\draw (y)  [fill=white] circle (\vr);
\draw (w1)  [fill=white] circle (\vr);
\draw (x)  [fill=black] circle (\vr);
\draw (z)  [fill=black] circle (\vr);
\draw (u)  [fill=black] circle (\vr);
\draw (xxx)  [fill=black] circle (\vr);
\draw (v)  [fill=black] circle (\vr);
\draw (w2)  [fill=black] circle (\vr);
\draw (w3)  [fill=black] circle (\vr);
\draw[rotate around={30:(-1.5,1.4)}] (-1.5,1.4) ellipse (80pt and 35pt);
\draw[rotate around={-30:(1.5,1.4)}] (1.5,1.4) ellipse (80pt and 35pt);
\draw[rotate around={45:(0.20,2.20)}] (0.20,2.20) ellipse (25pt and 15pt);
\draw[left] (u)++(-0.1,0) node {$u$};
\draw[right] (v)++(0.1,0) node {$v$};
\draw[above] (x)++(0,0.1) node {$x$};
\draw[above] (y)++(0,0.1) node {$y$};
\draw[above] (z)++(0,0.1) node {$z$};
\draw[right] (w2)++(0.1,0) node {$w_2$};
\draw[below] (w3)++(0,-0.10) node {$w_3$};
\draw[above] (w1)++(0,0) node {$w_1$};
\draw[rounded corners=15pt] (-4.2,-0.8) rectangle (4.3,3.6);
\draw (-3.5,3) node {$G$};
\draw (-4,4.5) node {$H:$};
\draw (-1.5,1.5) node {${\rm Ecc}_G(v)$};
\draw (1.5,1.5) node {${\rm Ecc}_G(u)$};
\draw [->] (2,3) -- (0.7,2.3);
\draw (2.2,3.2) node {$N[w_1]$};
\end{tikzpicture}
\end{center}
\caption{Construction from the proof of Theorem~\ref{thm:P3}}
\label{fig:diam-triples-P3}
\end{figure}
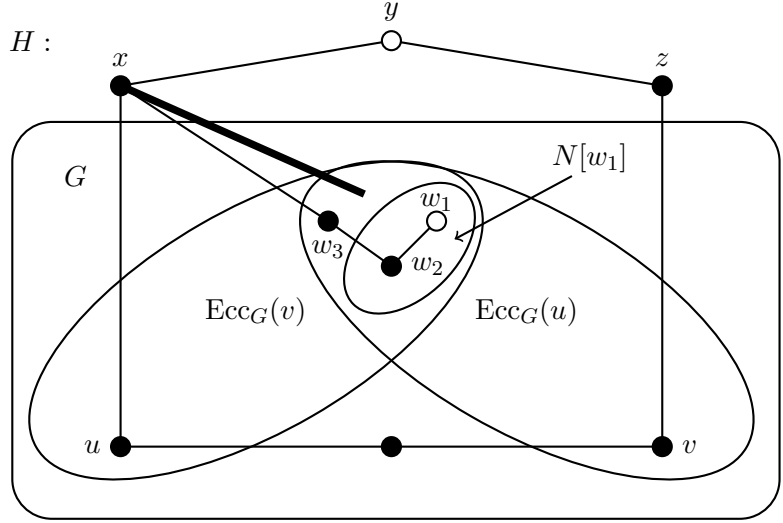

Thus the neighborhood of any vertex from $N_G[w_1]$ or not in $\left({\rm Ecc}_G(u)\cap {\rm Ecc}_G(v)\right)\cup \{u,v\}$ remains unchanged in $H$. From the construction of $H$ we first infer that $d_H(w_1,y)=4$ since $w_1w_2w_3xy$ is a shortest path between them. Any vertex from ${\rm Ecc}_G(u)\cap {\rm Ecc}_G(v)$, except $w_2$ and $w_1$, has eccentricity $3$ in $H$ because $z$ is its unique eccentric vertex and is at distance $3$. Also we have ${\rm ecc}_H(w_2)=3$ from the fact that ${\rm Ecc}_H(w_2)=\{y,z\}$ with $d_H(w_2,y)=3=d_H(w_2,z)$. Any vertex  from $N_G(u)\cap N_G(v)$ has eccentricity $3$ in $H$ because $y$ is its unique eccentric vertex and is at distance $3$.  Furthermore,  ${\rm ecc}_H(w)=3$ with ${\rm Ecc}_H(w)=\{y,z\}$ for any vertex $w\in N_G(u)\setminus N_G(v)$. By symmetry, we get ${\rm ecc}_H(w)=3$ for any vertex  $w\in N_G(v)\setminus N_G(u)$. Also we claim that ${\rm ecc}_H(x)=3={\rm ecc}_H(z)$ because $w_1\in {\rm Ecc}_H(x)$ and $w_2\in {\rm Ecc}_H(z)$, respectively, in $H$. Obviously we have ${\rm ecc}_H(u)=3={\rm ecc}_H(v)$ from ${\rm Ecc}_H(u)=\{z\}$ and ${\rm Ecc}_H(v)=\{x\}$, respectively. Therefore $H$ is a $3$-ASC embedding graph of $G$.
\end{proof}

Note that the Petersen graph  $PG$ fulfills the assumption of Theorem~\ref{thm:new-added}, hence $\theta_3(PG) = 3$. But  $PG$ does not fulfill the condition of Theorem~\ref{thm:isolated-vertex}. Further, $\theta_3(K_{1,3})=3$ holds by Theorem~\ref{thm:isolated-vertex}. However, $K_{1,3}$ does not fulfill the assumption of Theorem~\ref{thm:new-added} or that of Theorem~\ref{thm:P3}. Therefore, none of Theorems \ref{thm:new-added}, \ref{thm:isolated-vertex}, and \ref{thm:P3} gives a necessary condition for $\theta_3(G)=3$.

In view of Theorems~\ref{thm:isolated-vertex} and~\ref{thm:P3} one might wonder whether $\theta_3(G) =  3$ holds as soon as $G$ contains a diametrical triple. This is not the case as it can be verified on the graph $K_1 \oplus 3P_2$. More generally, we have the following result.

\begin{theorem}
\label{thm:union-of-complete}
If for any non-adjacent vertices $u, v$ the subgraph induced by ${\rm Ecc}_G(u)\cap {\rm Ecc}_G(v)$ is a disjoint union of complete graphs on at least two vertices  such that any outer neighbor of each vertex from any of these complete subgraphs belongs to $N_G(u)\cap N_G(v)$ and $G$ does not satisfy the assumption of Theorem~\ref{thm:new-added}, then $\theta_3(G)=4$.
\end{theorem}

\begin{proof}
By Lemma~\ref{lem:3-or-4} it suffices to prove that $\theta_3(G)>3$. Suppose on the contrary that $\theta_3(G)=3$ and let $H$ be a $3$-ASC graph obtained from $G$ by adding vertices $V_0 = \{x,y,z\}$ and some edges  (while keeping $G$ to be induced in $H$). We distinguish two cases.

\medskip\noindent
\textbf{Case 1.} No vertex from $V_0$ is pendant in $H$. \\
 We first claim that if $H[V_0]$ is complete or disconnected, then we can find some vertex with eccentricity $2$ in $H$, contradicting the $3$-ASC property of $H$.

 Suppose first that $H[V_0]$ is complete.  Then considering that $H$ has exactly two diametrical vertices with eccentricity $4$, we observe that one vertex, say $x$, from $V_0$ has eccentricity $4$ in $H$. Let $w\in V(G)$ be the other diametrical vertex in $H$. Then there is a vertex $w^{\prime}\in V(G)$ with $d_H(w^{\prime},x)=d_H(w^{\prime},w)=2$. It follows that ${\rm ecc}_H(w^{\prime})=2$.

 Assume next that $H[V_0]$ is disconnected. Then $H[V_0]\cong K_1\bigcup P_2$ or consists of three isolated vertices. In the former case, we can assume that $x$ is isolated in $H[V_0]$. Then $2\leq\min\{d_H(x,y),d_H(x,z)\}\leq 3$ since $H$ is a $3$-ASC graph. Without loss of generality assume that $\min\{d_H(x,y),d_H(x,z)\}=d_H(x,y)$. If $d_H(x,y)=2$, then there is a vertex $w\in V(G)$ as a common neighbor of $x$ and $y$ in $H$. Thus we get ${\rm ecc}_H(w)=2$. When $d_H(x,y)=3$, we assume that $xx^{\prime},yy^{\prime}\in E(H)$ where $x^{\prime},y^{\prime}\in V(G)$. Then $d_G(x^{\prime},y^{\prime})=1$ and ${\rm ecc}_H(y^{\prime})=2$. In the latter case, we conclude that there is a diametrical vertex in $V_0$ of $H$, say ${\rm ecc}_H(x)=4$. Then the other diametrical vertex of $H$ must be also in $V_0$. Otherwise, assume that $w\in V(G)$ has eccentricity $4$ in $H$. Then the distance from $w$ to any neighbor of $x$ in $G$ is $3$, which is impossible because $G$ has diameter $2$. Assume that ${\rm ecc}_H(y)=4$. Let $z^{\prime}$ be an arbitrary neighbor of $z$ in $H$ and $x^{\prime}$, $y^{\prime}$ the arbitrary neighbors of $x$ and $y$, respectively, in $H$. First we claim that neither of $x$ and $y$ belongs to ${\rm Ecc}_H(z)$. If not, without loss of generality let $x\in {\rm Ecc}_H(z)$. Considering that ${\rm ecc}_H(z^{\prime})=3$, we have $d_H(z^{\prime},x)=3$. Then $d_H(z,x)=4$, which is a clear contradiction. Thus we have $d_H(x,z)=2=d_H(y,z)$ from ${\rm ecc}_H(z)=3$. Now we conclude that $zx^{\prime},zy^{\prime}\in E(H)$. It follows that ${\rm ecc}_H(w)=2$ for any vertex $w\in N_G(x^{\prime})\bigcap N_G(y^{\prime})$.
This proves the claim.

So we are left with the situation when $H[V_0]\cong P_3$. Assume that the adjacency relation of $V_0$ in $H$ is $x-y-z$.  By the case assumption, there exists edges $xu,zv\in E(H)$.  Note first that $u\ne v$, for otherwise we would have ${\rm ecc}_H(u)=2$. The same conclusion holds if $d_H(u,v)=1$.

 Hence $d_H(u,v)=2$ must hold. By the theorem's assumption, ${\rm Ecc}_G(u)\cap {\rm Ecc}_G(v)$ induces a disjoint union of complete graphs $K_{n_k}$, $n_k\ge 2$. Note first that  $w$ is not adjacent to any vertex from $V_0$ for any vertex $w\in N_G(u)\cap N_G(v)$, for otherwise ${\rm ecc}_H(w)=2$ would hold.   Recall that $G$ has diameter $2$. If there is at least one vertex in each $K_{n_k}$ adjacent to  some vertex from $V_0$, then the distance is at most $3$ from any vertex from $G[{\rm Ecc}_G(u)\cap {\rm Ecc}_G(v)]$ to any vertex in $V_0$ of $H$. Thus any vertex in $H$ has eccentricity at most $3$,  a contradiction to the fact that $H$ is a $3$-ASC graph. Hence there must be some complete component $K_{n_i}$, from $G[{\rm Ecc}_G(u)\cap {\rm Ecc}_G(v)]$, such that none of vertices of $K_{n_i}$ is adjacent to any vertex from $V_0$ in $H$.  Let $w_1, w_2$ be arbitrary vertices of $K_{n_i}$. Recalling that any vertex from $V_0$ is not adjacent to any vertex from $N_G(u)\cap N_G(v)$ in $H$, from the theorem's condition that any outer neighbor of any vertex of these complete subgraphs belongs to $N_G(u)\cap N_G(v)$, we find that that $d_H(w_1,y) = d_H(w_2,y)= 4$.  We conclude that $H$ is not a $3$-ASC graph.

\medskip\noindent
\textbf{Case 2.} $V_0$ contains at least one pendant vertex of $H$. \\
Denote by $h$ the number of pendant vertices from $V_0$ in $H$. If $h=1$, we may assume that ${\rm deg}_H(z)=1$. Considering that $H$ is $3$-ASC graph, we claim that ${\rm ecc}_H(z)=4$. Assume that $z-z^{\prime}-w_0-w^{\prime}-w$ is a diametrical path in $H$. Then $d_H(z^{\prime},w)=3$, which implies that at most one vertex from $\{w,z^{\prime}\}$ belongs to $V(G)$. If $w_0\in V(G)$, then ${\rm ecc}_H(w_0)=2$, which is a contradiction. While $w_0\notin V(G)$, without loss of generality, we may assume that $w_0=x$. Next we divide into the following two subcases. For $z^{\prime}\in V(G)$, recalling that $z^{\prime},w$ can not belong to $V(G)$ simultaneously, we have $w\notin V(G)$, i.e., $w=y$. In this subcase, we claim that $z^{\prime}$ and $w^{\prime}$ forms a diametrical pair in $G$. Let $w^{*}\in N_G(w^{\prime})\cap N_G(z^{\prime})$. Then the distance from $w^*$ to any vertex in $V_0$ is at most $2$ in $H$. Recalling that $G$ has diameter $2$. Then ${\rm ecc}_H(w^{*})=2$, which is also impossible. If $z^{\prime}\notin V(G)$, then $z^{\prime}=y$. So the adjacency relation of $V_0$ in $H$ is $x-y-z$ with $z$ being pendant. If ${\rm deg}_G(w^{\prime})=n-1$, then ${\rm ecc}_H(x)=2$ clearly. If ${\rm ecc}_G(w^{\prime})=2$, we consider the vertices from ${\rm Ecc}_G(w^{\prime})$. Now we conclude that $w_1$ is adjacent to $x$ or $y$ in $H$ for any vertex $w_1\in {\rm Ecc}_G(w^{\prime})$. Otherwise, $d_H(w_1,z)> 4$, contradicting to the $3$-ASC property of $H$. Therefore, we get ${\rm ecc}_H(x)=2$, a clear contradiction comes again.

If $h=2$, without loss of generality, we assume that $x,y$ are both pendant in $H$. Clearly, ${\rm eec}_H(x)=4={\rm eec}_H(y)$. Let $xx^{\prime},yy^{\prime}\in E(H)$. Since $H$ is $3$-ASC graph, there is a diametrical path $x-x^{\prime}-w-y^{\prime}-y$ with $w\in V(G)\cup \{z\}$. Then, considering that $G$ does not satisfy the assumption of Theorem~\ref{thm:new-added}, we have ${\rm ecc}_H(w)=2$. This is an obvious contradiction. If $h=3$, then each of $x,y$ and $z$ has eccentricity $4$ in $H$, contradicting the definition of $3$-ASC graph.
\end{proof}

Consider the graph $G$ shown in Fig.~\ref{fig:strange-example}. It is of diameter $2$ and fulfils the first condition of Theorem~\ref{thm:union-of-complete}. For the selected vertices $u$ and $v$ we have ${\rm Ecc}_G(u)\cap {\rm Ecc}_G(v) = \{x,y\}$. Note that each of these vertices has an outer neighbor that does not belong to $N_G(u)\cap N_G(v)$. Hence the second condition of the theorem is not fulfilled. In the figure an embedding of $G$ into a $3$-ASC graph is also shown, therefore $\theta_3(G) = 3$. In summary, the technical condition about the outer neighbors from Theorem~\ref{thm:union-of-complete} cannot be avoided.

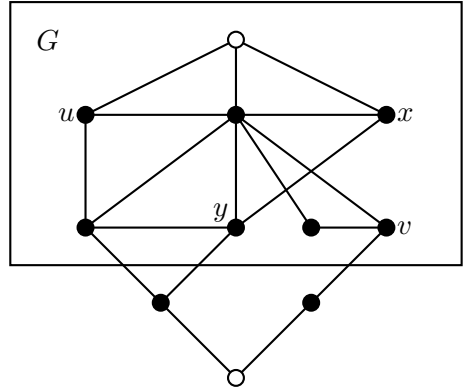
\begin{figure}[ht!]
\begin{center}
\begin{tikzpicture}[scale=1,style=thick]
\def\vr{3pt}
\path (2,0) coordinate (a);
\path (1,1) coordinate (b);
\path (3,1) coordinate (c);
\path (0,2) coordinate (d);
\path (2,2) coordinate (e);
\path (3,2) coordinate (f);
\path (4,2) coordinate (v);
\path (0,3.5) coordinate (u);
\path (2,3.5) coordinate (g);
\path (4,3.5) coordinate (h);
\path (2,4.5) coordinate (i);
\draw (a) -- (b) -- (d) -- (u) -- (i) -- (h) -- (e) -- (b);
\draw (a) -- (c) -- (v) -- (f) -- (g) -- (d) -- (e);
\draw (v) -- (g) -- (i);
\draw (u) -- (g) -- (h);
\draw (g) -- (e);
\draw (-1,1.5) rectangle (5,5);
\draw (a) [fill=white] circle (\vr);
\draw (b) [fill=black] circle (\vr);
\draw (c) [fill=black] circle (\vr);
\draw (d) [fill=black] circle (\vr);
\draw (e) [fill=black] circle (\vr);
\draw (f) [fill=black] circle (\vr);
\draw (g) [fill=black] circle (\vr);
\draw (h) [fill=black] circle (\vr);
\draw (i) [fill=white] circle (\vr);
\draw (u) [fill=black] circle (\vr);
\draw (v) [fill=black] circle (\vr);
\draw [left] (u) node {$u$};
\draw [right] (v) node {$v$};
\draw [right] (h) node {$x$};
\draw (1.8,2.2) node {$y$};
\draw (-0.5,4.5) node {$G$};
\end{tikzpicture}
\end{center}
\caption{Graph $G$ with $\theta_3(G) = 3$}
\label{fig:strange-example}
\end{figure}

Recall that the well-known cocktail party graph $CP(n)$ is a complete graph on $2n$ vertices with a $1$-factor removed. By Theorem~\ref{thm:no-triple}, $\theta_3(CP(n))=4$ but $CP(n)$ does not fulfill the condition of Theorem~\ref{thm:union-of-complete}. In addition, from Theorem~\ref{thm:union-of-complete} we get that $\theta_3(K_1\oplus tP_2)=4$ holds for any $t\ge 3$, but $K_1\oplus tP_2$, $t\ge 3$, does not fulfill the condition of Theorem \ref{thm:no-triple}. Therefore, Theorems~\ref{thm:no-triple} and \ref{thm:union-of-complete} give sufficient but not necessary conditions for $G$ having $\theta_3(G)=4$.

\section{$3$-ASC index of some graphs with large diameter}
\label{sec:large-diam}

 In this section we will determine the $3$-ASC index of paths  of order $n\geq 1$, cycles of order $n\geq 3$, and trees of order  $n\geq 10$ and diameter $n-2$.

\subsection{Paths}

 In this subsection we will always assume that $V(P_n)=\{v_1,\ldots,v_{n}\}$ and $E(P_n)=\{v_iv_{i+1}:\ i=1,\ldots, n-1\}$. We will also use the graph  $C_{2k}^{*}$, $k\ge 2$, which is obtained from $C_{2k}$ by attaching a pendant vertex to one of its vertices. As observed in~\cite{KNW2011}, the graph $C_{2r}^{*}$ is an $r$-ASC graph. First we deal with the cases when $n$ is small.

 The cases $n=1$ and $n=2$ are covered by Theorem~\ref{thm-complete}. We have already observed that a smallest $3$-ASC graph is of order $7$. Since each of the paths $P_3$, $P_4$, $P_5$, and $P_6$ is an induced subgraph of $C_6^*$, it follows that $\theta_3(P_3)=4$, $\theta_3(P_4)=3$, $\theta_3(P_5)=2$, and $\theta_3(P_6)=1$. For $n=7$ and $n=8$,  consider the graphs $H_1$ and $H_2$ from Fig.~\ref{ADD}. By a straightforward checking we find that both $H_1$ and $H_2$ are $3$-ASC graphs. Therefore $\theta_3(P_7)=\theta_3(P_8)=1$.

\begin{figure}[ht!]
\begin{center}
\begin{tikzpicture}[scale=1,style=thick]
\def\vr{3pt}

\draw (-7,0) -- (-6.2,0) -- (-5.4,0) -- (-4.6,0) -- (-3.8,0) -- (-3,0) -- (-2.2,0);
\draw (0.6,0) -- (1.4,0) --  (2.2,0) -- (3,0) -- (3.8,0) -- (4.6,0) -- (5.4,0) -- (6.2,0);
\draw (-7,0) -- (-5,1) -- (-6.2,0);
\draw (-5,1) -- (-5.4,0);
\draw (-2.2,0) -- (-5,1);
\draw (0.6,0) -- (3,1) -- (2.2,0);
\draw (1.4,0) -- (3,1) -- (5.4,0);

\draw (6.2,0) -- (3,1);

\draw (-7,0)  [fill=white] circle (\vr);
\draw (-6.2,0)  [fill=black] circle (\vr);
\draw (-5.4,0)  [fill=black] circle (\vr);
\draw (-4.6,0)  [fill=black] circle (\vr);
\draw (-3.8,0)  [fill=white] circle (\vr);
\draw (-3,0)  [fill=black] circle (\vr);
\draw (-2.2,0)  [fill=black] circle (\vr);

\draw (0.6,0)  [fill=white] circle (\vr);
\draw (1.4,0)  [fill=black] circle (\vr);
\draw (2.2,0)  [fill=black] circle (\vr);
\draw (3,0)  [fill=black] circle (\vr);
\draw (3.8,0)  [fill=white] circle (\vr);
\draw (4.6,0)  [fill=black] circle (\vr);
\draw (5.4,0)  [fill=black] circle (\vr);
\draw (6.2,0)  [fill=black] circle (\vr);

\draw (-5,1)  [fill=black] circle (\vr);

\draw (3,1)  [fill=black] circle (\vr);

\draw (-5,1.4) node {$x$};

\draw (3,1.4) node {$x$};

\draw (-7,-0.3) node {$v_1$};
\draw (-6.2,-0.3) node {$v_2$};
\draw (-5.4,-0.3) node {$v_3$};
\draw (-4.6,-0.3) node {$v_4$};
\draw (-3.8,-0.3) node {$v_5$};
\draw (-3,-0.3) node {$v_6$};
\draw (-2.2,-0.3) node {$v_7$};

\draw (0.6,-0.3) node {$v_1$};
\draw (1.4,-0.3) node {$v_2$};
\draw (2.2,-0.3) node {$v_3$};
\draw (3,-0.3) node {$v_4$};
\draw (3.8,-0.3) node {$v_5$};
\draw (4.6,-0.3) node {$v_6$};
\draw (5.4,-0.3) node {$v_7$};
\draw (6.2,-0.3) node {$v_8$};

\draw (-4.6,-0.8) node {$H_1$};
\draw (3.8,-0.8) node {$H_2$};

\end{tikzpicture}
\end{center}
\caption{ The graphs $H_1$ and $H_2$}
\label{ADD}
\end{figure}
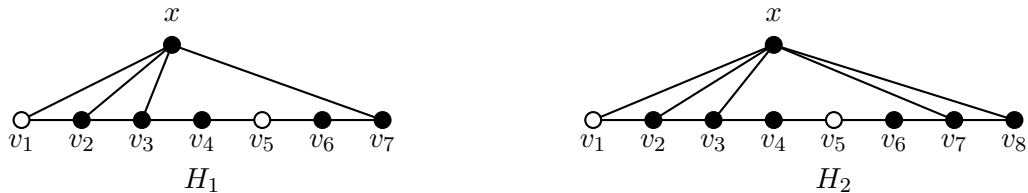

In the rest we may thus assume that $n\ge 9$ and first prove:

\begin{lemma}
\label{L2.4}
 If $n\geq 9$, then $\theta_3(P_n)\le2$.
\end{lemma}

\begin{proof}
For $n=9$ the assertion follows from the embedding presented in Fig.~\ref{Q2}.

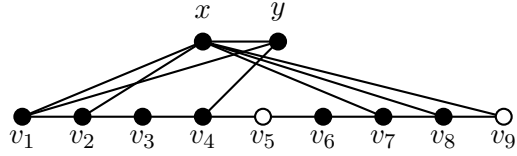
\begin{figure}[ht!]
\begin{center}
\begin{tikzpicture}[scale=1,style=thick]
\def\vr{3pt}

\draw (0.6,0) -- (1.4,0) --  (2.2,0) -- (3,0) -- (3.8,0) -- (4.6,0) -- (5.4,0) -- (6.2,0) -- (7,0);

\draw (0.6,0) -- (3,1) -- (4,1) -- (3,0);
\draw (1.4,0) -- (3,1) -- (5.4,0);
\draw (0.6,0) -- (4,1);
\draw (6.2,0) -- (3,1) -- (7,0);

\draw (0.6,0)  [fill=black] circle (\vr);
\draw (1.4,0)  [fill=black] circle (\vr);
\draw (2.2,0)  [fill=black] circle (\vr);
\draw (3,0)  [fill=black] circle (\vr);
\draw (3.8,0)  [fill=white] circle (\vr);
\draw (4.6,0)  [fill=black] circle (\vr);
\draw (5.4,0)  [fill=black] circle (\vr);
\draw (6.2,0)  [fill=black] circle (\vr);
\draw (7,0)  [fill=white] circle (\vr);

\draw (3,1)  [fill=black] circle (\vr);
\draw (4,1)  [fill=black] circle (\vr);

\draw (3,1.4) node {$x$};
\draw (4,1.4) node {$y$};

\draw (0.6,-0.3) node {$v_1$};
\draw (1.4,-0.3) node {$v_2$};
\draw (2.2,-0.3) node {$v_3$};
\draw (3,-0.3) node {$v_4$};
\draw (3.8,-0.3) node {$v_5$};
\draw (4.6,-0.3) node {$v_6$};
\draw (5.4,-0.3) node {$v_7$};
\draw (6.2,-0.3) node {$v_8$};
\draw (7,-0.3) node {$v_9$};
\end{tikzpicture}
\end{center}
\caption{ Embedding of $P_9$ into a $3$-ASC graph of order $11$}
\label{Q2}
\end{figure}

For $n\geq 10$ we construct a graph  $G_n$ schematically shown in Fig.~\ref{Q3}, where all the vertices $v_{10},\ldots, v_{n}$ are adjacent to both $x$ and $y$.

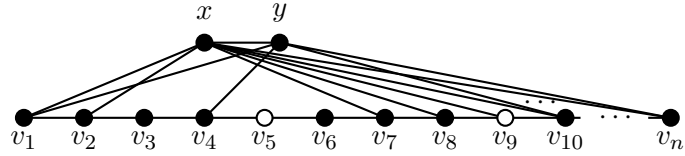
\begin{figure}[ht!]
\begin{center}
\begin{tikzpicture}[scale=1,style=thick]
\def\vr{3pt}

\draw (0.6,0) -- (1.4,0) --  (2.2,0) -- (3,0) -- (3.8,0) -- (4.6,0) -- (5.4,0) -- (6.2,0) -- (7,0) -- (7.8,0) -- (8,0);
\draw (8.9,0) -- (9.2,0);
\draw (7.8,0) -- (3,1) -- (9.2,0);
\draw (7.8,0) -- (4,1) -- (9.2,0);
\draw (0.6,0) -- (3,1) -- (4,1) -- (3,0);
\draw (1.4,0) -- (3,1) -- (5.4,0);
\draw (0.6,0) -- (4,1);
\draw (6.2,0) -- (3,1) -- (7,0);

\draw (0.6,0)  [fill=black] circle (\vr);
\draw (1.4,0)  [fill=black] circle (\vr);
\draw (2.2,0)  [fill=black] circle (\vr);
\draw (3,0)  [fill=black] circle (\vr);
\draw (3.8,0)  [fill=white] circle (\vr);
\draw (4.6,0)  [fill=black] circle (\vr);
\draw (5.4,0)  [fill=black] circle (\vr);
\draw (6.2,0)  [fill=black] circle (\vr);
\draw (7,0)  [fill=white] circle (\vr);
\draw (7.8,0)  [fill=black] circle (\vr);
\draw (9.2,0)  [fill=black] circle (\vr);

\draw (3,1)  [fill=black] circle (\vr);
\draw (4,1)  [fill=black] circle (\vr);

\draw (3,1.4) node {$x$};
\draw (4,1.4) node {$y$};

\draw (0.6,-0.3) node {$v_1$};
\draw (1.4,-0.3) node {$v_2$};
\draw (2.2,-0.3) node {$v_3$};
\draw (3,-0.3) node {$v_4$};
\draw (3.8,-0.3) node {$v_5$};
\draw (4.6,-0.3) node {$v_6$};
\draw (5.4,-0.3) node {$v_7$};
\draw (6.2,-0.3) node {$v_8$};
\draw (7,-0.3) node {$v_9$};
\draw (7.8,-0.3) node {$v_{10}$};
\draw (9.2,-0.3) node {$v_{n}$};
\draw (8.5,0) node {$\cdots$};
\draw (7.5,0.2) node {$\cdots$};
\end{tikzpicture}
\end{center}
\caption{Graphs $G_n$}
\label{Q3}
\end{figure}

It is straightforward to check that $G_n$ is  a $3$-ASC embedding graph of $P_n$.
\end{proof}

\begin{lemma}
\label{L2.5}
If $n\geq 9$, then $\theta_3(P_n)\geq 2$.
\end{lemma}

\begin{proof} Since paths $P_n$ are not $3$-ASC graphs we have $\theta_3(P_n)> 0$ for  all $n\geq 9$. In the following we prove that $\theta_3(P_n)>1$  for $n\ge 9$. Assume on the contrary that there exists a $3$-ASC graph $G$ with $V(G)=V(P_n)\cup\{x\}$ that contains $P_n$ as an induced subgraph.

First assume that $x$ is a diametrical vertex in $G$  and let $v_i$ be its diametrical pair,  so that $d_G(x,v_i)=4$. If $i=1$, then  clearly $xv_1, xv_2, xv_3\notin E(G)$ and $xv_4\in E(G)$. Since $P_n$ is an induced subgraph of $G$, we conclude that $d_G(v_1,v_j)\geq 4$ for $j\geq 5$, which contradicts the fact that $G$ is a $3$-ASC graph. Similarly we  get a contradiction if $i = 2$. By symmetry, $i\neq n-1$ and $i\neq n$. Finally, if $2<i< n-2$, then $x$ is not adjacent to any vertex of $v_{i-2}$, $v_{i-1}$, $v_{i+1}$ and $v_{i+2}$. Thus $d_G(v_{i-2},v_{i+2})=4$, another contradiction.

It follows that  the two diametrical vertices of $G$ are on $P_n$, say $v_i$ and $v_j$, where $i<j$. Then $d_G(v_i,v_j)=4$. Clearly, $xv_i$ and $xv_j$ can not belong to $E(G)$ simultaneously. We distinguish the following two cases.

\medskip\noindent
\textbf{Case 1.}  Either $xv_i\in E(G)$ or $xv_j\in E(G)$.
\\
We may without loss of generality assume that $xv_i\in E(G)$.  If $i>1$, then $d_G(v_{i-1}, v_j)\ge 4$, which implies that $G$ has as least three diametrical vertices. Hence $i=1$ must hold.  Since $xv_1\in E(G)$ and $d_G(v_1,v_j)=4$, neither of $v_{j}$ and $v_{j-1}$ is adjacent to $x$ in $G$. If $j<n$, then the same conclusion holds also for $v_{j+1}$. In any case, $d_G(v_{j-4},v_j) = 4$. This is only possible if $j-4 = i = 1$, that is, if $j=5$. As already argued above, $xv_4, xv_5, xv_6\notin E(G)$, for otherwise $d_G(v_1,v_5)\le 3$ would hold. But then it follows that $d_G(v_9, v_5) = 4$, so we would again have at least three diametrical vertices. Case 1 hence cannot happen.

\medskip\noindent
\textbf{Case 2.} Neither of $v_i$ and $v_j$ is adjacent to $x$ in $G$. \\
 Suppose first that $i\ge 2$ and $j\le n-1$. Then $x$ must be adjacent to each vertex from $V_1=\{v_{k}:\ 1\leq k<i\}$  because $d_G(v_k,v_j)\ge 4$ holds if $v_kx\notin E(G)$. Analogously, $x$ must be adjacent to each vertex from  $V_2=\{v_t:\ j<t\leq n\}$.  Since $d_G(v_i,v_j)=4$ we clearly have that $j-i\geq 4$. If $j-i=4$, then $xv_{i+2}\notin E(G)$,  for otherwise ${\rm ecc}_G(x)=2$. Considering that $d_G(v_{i+3},v_k)\leq 3$ for any vertex $v_k\in V_1$, we have $xv_{i+3}\in E(G)$. By symmetry, we get $xv_{i+1}\in E(G)$. Hence  ${\rm ecc}_G(x)=2$ holds again, a contradiction. If $j-i>4$, then from ${\rm ecc}_G(x)=3$ we find out that there are three consecutive vertices $v_s,v_{s+1},v_{s+2}$ where $i\leq s<s+2\leq j$ such that each of them has degree $2$ in $G$. But now there must be a vertex $v_p\in V_1\cup V_2$ such that $d_G(v_{s+1},v_{p})=4$.

Let next $i=1$ and $j<n$. Then $x$ must be adjacent to all vertices $v_t$, $t>j$, because otherwise $d_G(v_t,v_1)\ge 4$ would hold for such a vertex $v_t$. Assume that $xv_2\notin E(G)$. Then $j=5$ because otherwise $d_G(v_2,v_{j}) > 3$.  But then $d_G(v_1,v_{6}) > 3$, a contradiction. Hence necessary $xv_2\in E(G)$.  As above we now find consecutive vertices $v_s,v_{s+1},v_{s+2}$ such that none of them is adjacent to $x$ in $G$ where $3\leq s<s+2\leq j$, for otherwise ${\rm ecc}_G(x)=2$. But then $d_G(v_{s+1},v_{1}) > 3$ for $s>3$, or $d_G(v_{s+1},v_p)>3$ where $p>j$ for $s=3$ from $n\geq 9$, another contradiction.

The case when $i> 1$ and $j=n$ is symmetric to the last case.  Hence we are left with the situation when $i=1$ and $j=n$. If $xv_2\notin E(G)$, then we have $d_G(v_1,v_{n-1}) > 3$. So $xv_2\in E(G)$ and by a parallel argument also $xv_{n-1}\in E(G)$. Now, as above, we have a triple of vertices $v_s,v_{s+1},v_{s+2}$ none of which is adjacent to $x$ in $G$. But then $d_G(v_{s+1},v_1) > 3$ or $d_G(v_{s+1},v_n) > 3$, a final contradiction.\end{proof}

 Combining Lemmas~\ref{L2.4}, \ref{L2.5} with the arguments given before we have the following result.

\begin{theorem}\label{thm:paths}  The $3$-ASC index of paths is:
$$\theta_3(P_n)=\left\{
\begin{array}{ll}
7-n, & 1\le n\le 5\,; \\
1,& 6\le n \le 8\,; \\
2,& n\geq 9\,.
\end{array}
\right.$$
\end{theorem}

\subsection{Cycles}

Next we determine the $3$-ASC index of cycles. The triangle $C_3$ is a complete graph, so $\theta_3(C_3)=5$ by Theorem~\ref{thm-complete}.  We have already observed that $\theta_3(C_4)=4$ and $\theta_3(C_5)=3$. Recall that  $C_6^*$ (the graph obtained by attaching a pendant vertex to one of the vertices of $C_6$) is a $3$-ASC graph, hence $\theta_3(C_6)=1$. Recall further that $C_{2r+1}'$ is the graph which consists of a cycle $C_{2r+1}$ and a vertex adjacent to exactly two consecutive vertices in $C_{2r+1}$ and that $C_{2r+1}'$ is an $r$-ASC graph. Thus $C_7'$ is $3$-ASC graph and $\theta_3(C_7)=1$ holds accordingly.  Let in addition $C_8''$  be the graph obtained by joining five consecutive vertices of $C_8$ with a new vertex. It can be easily checked that $C_8''$ is a $3$-ASC graph  and thus $\theta_3(C_8)=1$. The  result for all cycles reads as follows.

\begin{theorem}\label{thm:cycles}  The $3$-ASC index of cycles is:
$$\theta_3(C_n)=\left\{
\begin{array}{ll}
8-n, & 3\le n \le 5\,; \\
1,& 6\le n \le 8\,; \\
2,& n\geq 9\,.
\end{array}
\right.$$
\end{theorem}

\begin{proof}
The result for $n\le 8$ has been established above, hence in the following we will assume that $n\geq 9$. Let $V(C_n)=\{v_1,\ldots,v_n\}$, $E(C_n)=\{v_iv_{i+1}|i=1,\ldots,n-1\}\cup \{v_nv_1\}$, and for a positive integer $k$, $1\leq k\leq n$, we set $|k|_{n}=\min\{k,n-k\}$.

We first prove that $\theta_3(C_n)\leq 2$. For $n=9$ consider the graph $G$ shown in Fig.~\ref{Q4}. It can be routinely checked that $G$ is a $3$-ASC graph ( recall that two vertices filled white are diametrical vertices as mentioned before).

\begin{figure}[ht!]
\begin{center}
\begin{tikzpicture}[scale=1,style=thick]
\def\vr{3pt}

\draw (0.6,-0.4) -- (-0.1,-0.4) --  (-0.3,0.4) -- (-0.7,1.2) -- (0.2,2) -- (1.1,2) -- (2.1,1.2) -- (1.6,0.4) -- (1.3,-0.4) -- (0.6,-0.4);
\draw (0.6,1.1) -- (0.8,0.6);
\draw (-0.3,0.4) -- (0.6,1.1) -- (-0.7,1.2);
\draw (2.1,1.2) -- (0.6,1.1) -- (1.1,2);
\draw (0.2,2) -- (0.6,1.1);
\draw (1.3,-0.4) -- (0.8,0.6) -- (1.1,2);

\draw (0.6,-0.4)  [fill=white] circle (\vr);
\draw (1.3,-0.4)  [fill=black] circle (\vr);
\draw (-0.1,-0.4)  [fill=black] circle (\vr);
\draw (1.6,0.4)  [fill=black] circle (\vr);
\draw (-0.3,0.4)  [fill=black] circle (\vr);
\draw (2.1,1.2)  [fill=black] circle (\vr);
\draw (-0.7,1.2)  [fill=black] circle (\vr);
\draw (1.1,2)  [fill=black] circle (\vr);
\draw (0.2,2)  [fill=white] circle (\vr);
\draw (0.6,1.1)  [fill=black] circle (\vr);
\draw (0.8,0.6)  [fill=black] circle (\vr);

\end{tikzpicture}
\end{center}
\caption{ The graph $G$}
\label{Q4}
\end{figure}
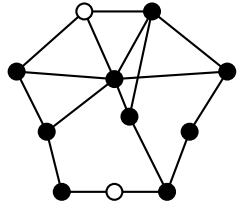

When $n\geq 10$, we construct the graph  $H_n$ as shown in Fig.~\ref{Q5}.  In $H_n$ both $x$ and $y$ are adjacent to any vertex from $\{v_k:\ 10\leq k\leq n\}$.  The graph $H_n$ is a $3$-ASC embedding graph of $C_n$, completing the proof of upper bound on $\theta_3(C_n)$.

\begin{figure}[ht!]
\begin{center}
\begin{tikzpicture}[scale=1,style=thick]
\def\vr{3pt}

\draw (0.6,-0.4) -- (-0.3,-0.4) --  (-0.4,0.4) -- (-0.7,1.2) -- (0,2.1);
\draw (1.4,2.1) -- (2,1.6) -- (2.2,1) -- (1.9,0.4) -- (1.7,0) -- (1.5,-0.4) -- (0.6,-0.4);
\draw (0.5,1) -- (1.2,0.5);
\draw (-0.4,0.4) -- (0.5,1) -- (-0.7,1.2);
\draw (2.2,1) -- (0.5,1) -- (1.4,2.1);
\draw (0,2.1) -- (0.5,1) -- (2,1.6);
\draw (1.7,0) -- (1.2,0.5) -- (1.4,2.1);
\draw (-0.3,-0.4) -- (0.5,1) -- (1.4,2.1);
\draw (2,1.6) -- (1.2,0.5) -- (0,2.1);
\draw (0,2.1) -- (0.2,2.1);
\draw (1.4,2.1) -- (1.2,2.1);

\draw (0.6,-0.4)  [fill=black] circle (\vr);
\draw (1.5,-0.4)  [fill=white] circle (\vr);
\draw (1.7,0)  [fill=black] circle (\vr);
\draw (-0.3,-0.4)  [fill=black] circle (\vr);
\draw (1.9,0.4)  [fill=black] circle (\vr);
\draw (-0.4,0.4)  [fill=black] circle (\vr);
\draw (2,1.6)  [fill=black] circle (\vr);
\draw (2.2,1)  [fill=black] circle (\vr);
\draw (-0.7,1.2)  [fill=white] circle (\vr);
\draw (1.4,2.1)  [fill=black] circle (\vr);
\draw (0,2.1)  [fill=black] circle (\vr);
\draw (0.5,1)  [fill=black] circle (\vr);
\draw (1.2,0.5)  [fill=black] circle (\vr);

\draw (0.5,0.7) node {$x$};
\draw (1.1,0.2) node {$y$};
\draw (0,2.4) node {$v_{10}$};
\draw (1.4,2.4) node {$v_n$};
\draw (0.7,2.1) node {$\cdots\cdots$};
\draw (0.7,-1) node {$H_n$};
\end{tikzpicture}
\end{center}
\caption{ The graph $H_n$}
\label{Q5}
\end{figure}
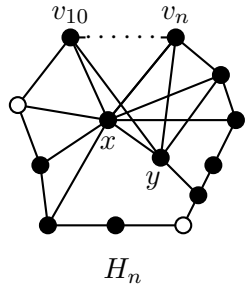

 To complete the proof we need to show that $\theta_3(C_n)\geq 2$. As $C_n$, $n\geq 9$, is not a $3$-ASC graph, we only need to prove that $\theta_3(C_n)> 1$.  Suppose on the contrary that a graph $G$ with vertex set $V(C_n)\cup \{x\}$ is a $3$-ASC embedding graph of $C_n$.

If ${\rm ecc}_G(x)=4$, then there exists an induced internal $(v_i,v_j)$-path $P$ (that is, a path whose internal vertices are of degree $2$) of length $6$ in $G$  such that $xv_i,xv_j\in E(G)$ and $d_G(x,v_{t})=4$ where $v_t\in V(P)$ with $|t-i|_n=|t-j|_n=3$. Thus we conclude that $d_G(v_{t},v_s)\geq 4$ for $v_s\in V(C_n)\setminus V(P)$. This is an apparent contradiction to the $3$-ASC property of $G$. If ${\rm ecc}_G(x)=3$, then,  considering a vertex being at distance $3$ from $x$, we infer that there is an induced internal $(v_i,v_j)$-path $P$ with $|j-i|_n=4$ or $5$ and $v_k\in V(P)$ such that  both $v_i$ and $v_j$ are adjacent to $x$ in $G$ and that neither of $v_i$ and $v_j$ is adjacent to $v_k$. Now we consider the eccentricity of $v_k$ in $G$.  If ${\rm ecc}_G(v_k)=3$, then any vertex from $V(C_n)\setminus V(P)$ is adjacent to $v_i$ or $v_j$. However, this is impossible when $n\geq 9$. Then we get  ${\rm ecc}_G(v_k)=4$. Without loss of generality, assume that $i<j$ and $v_i$ has a neighbor $v_{i-1}\in V(C_n)\setminus V(P)$ and $v_j$ has a neighbor $v_{j+1}\in V(C_n)\setminus V(P)$. Let $V_0=V(C_n)\setminus \left(V(P)\cup \{v_{i-1},v_{j+1}\}\right)$. Since $n\geq 9$, we have $|V_0|\geq 2$ for $|j-i|_n=4$ or $|V_0|\geq 1$ for $|j-i|_n=5$. Therefore any vertex in $V_0$ must have eccentricity at least $4$ in $G$. Moreover, when $|j-i|_n=5$, $v_{i-1}$ or $v_{j+1}$ has eccentricity $4$ in $G$. This contradicts the fact that $G$ is a $3$-ASC graph  and we are done.
\end{proof}

\subsection{Trees of order $n$ and diameter $n-2$}

Now we turn to $\theta_3(T)$ for trees $T$ of order $n$ and diameter $n-2$.  Any such tree can be obtained by attaching a pendant vertex to any non-pendant vertex of the path $P_{n-1}$.  Throughout this subsection let $V(T)=\{v_1,\ldots,v_n\}$ and $E(T)=\{v_iv_{i+1}:\ i=1,\ldots,n-2\}\cup \{v_kv_n\}$,  where $k$ is a fixed integer from $\{2,\ldots, n-2\}$.

\begin{lemma}\label{Le3.3}
If $T$ be a tree of order $n\geq 10$ and diameter $n-2$, then $\theta_3(T)\leq 2$.
\end{lemma}

\begin{proof}
 Based on the value of $k$ introduced above we distinguish two cases.

\medskip\noindent
\textbf{Case 1.} $k=5$.\\
 Let $G$ be the graph obtained from $T\cup P_2$, where $P_2=xy$, by adding edges $xv_i$,  $i\in \{1,2\}\cup \{7,\ldots, n-1\}$ and $yv_j$,  $j\in \{1,4\}\cup \{10,\ldots, n-1,n\}$. Then it is straightforward to verify that $G$ is a $3$-ASC embedding graph of $T$ where $v_5,v_9$ are two diametrical vertices in  $G$.

\medskip\noindent
\textbf{Case 2.} $k\neq 5$. \\
In this case we can construct a similar graph $H$ as in Case 1, but now one more edge $xv_n$ is added to $H$ (except that $yv_n\notin E(G)$ for $k=6$ otherwise ${\rm ecc}_G(y)=2$). It is not difficult to check that $H$ is a $3$-ASC embedding graph of $T$ with $v_5,v_9$  being its unique diametrical vertices.
\end{proof}

\begin{lemma}\label{Le3.4}
If $T$ is a tree of order $n\geq 10$ and diameter $n-2$, then $\theta_3(T)\geq 2$.
\end{lemma}

\begin{proof} Since $T$ is clearly not a $3$-ASC graph, it suffices to prove that $\theta_3(T)>1$. Suppose on the contrary that there exists a $3$-ASC embedding graph $G$ of $T$ with $V(G)=V(T)\cup\{x\}$. We distinguish the following two cases.

\medskip\noindent
\textbf{Case 1.} $xv_n\notin E(G)$.\\
In this case  $v_n$ is pendant in $G$ with $v_kv_n\in E(G)$. Note that $G$ is a $3$-ASC  graph. Then ${\rm ecc}_G(v_n)=4$ and ${\rm ecc}_G(v_k)=3$.

Suppose first that $xv_k\in E(G)$.  Let $v_i$ be the vertex with $d(v_n,v_i)=4$. Clearly, $d_T(v_k,v_i)\geq 3$ and $xv_i\notin E(G)$. If $v_j\in V(T)\setminus \{v_n,v_i\}$, then $d_G(v_n,v_j)\leq 3$ and $d_G(v_i,v_j)\leq 3$  and consequently $d_G(v_k,v_j)\leq 2$. Obviously, any vertex $v_j$ $(j\neq n, i)$ with $d_T(v_j,v_k)\geq 3$ must be adjacent to $x$ in $G$. If $d_T(v_i,v_k)>3$, then we get ${\rm ecc}_G(x)=2$, which is impossible because of the $3$-ASC property of $G$. Hence $d_T(v_k,v_i)=3$ and without loss of generality we may assume that $k<i$. First we claim that $xv_s\in E(G)$ for any vertex $v_s$ with $s<k$ and $d_T(v_k,v_s)\leq 2$. Otherwise, we have $d_G(v_s,v_i)=4$. Then ${\rm ecc}_G(v_s)\geq 4$, contradicting to the $3$-ASC property of $G$. Since ${\rm ecc}_G(x)=3$,  any vertex in the $(v_i,v_k)$-path except $v_k$ has degree $2$ in $G$. Then any vertex $v_p$ with $d_T(v_p,v_{k+2})>3$ has the property that $d_G(v_{k+2},v_p)=4$. Clearly, ${\rm ecc}_G(v_{k+2})$ and ${\rm ecc}_G(v_p)$ are more than $4$, contradicting to the $3$-ASC property of $G$, again.

Now we consider the subcase when $xv_k\notin E(G)$. From ${\rm ecc}_G(v_n)=4$, we find that there is at most one vertex $v_p$ with $d_G(v_n,v_p)=4$ and $d_G(v_n,v_j)\leq 3$ for any vertex $v_j$ with $j\neq p$. Considering that $v_n$ is pendant in $G$, all vertices except $v_p$ in $T$ are at distance at most $2$ from $v_k$ in $G$. But, combining the fact that $xv_k\notin E(G)$, we have $d_G(v_j,v_k)\geq 3$ for any vertex $v_j$ with $d_T(v_j,v_k)\geq 3$ with $j\neq p$. A contradiction occurs.

\medskip\noindent
\textbf{Case 2.} $xv_n\in E(G)$.\\
In this case we first assume that ${\rm ecc}_G(v_n)=4$. Then ${\rm ecc}_G(v_k)=3$ and there exists a vertex $v_i$ with $d_G(v_n,v_i)=4$, i.e., $d_G(v_k,v_i)=3$ and $d_G(x,v_i)=3$. If $i=1$, we find that $d_G(v_i,v_j)\geq 4$ for any vertex $v_j$ with $k+1\leq j\leq n-1$, which is impossible since $G$ is a $3$-ASC graph. By symmetry, $i\neq n-1$.  It follows that $v_i$ has two neighbors $v_{i-1}$ and $v_{i+1}$ in $T$. Moreover, neither of $v_{i-1}$ and $v_{i+1}$ is adjacent to $x$ in $G$.

Now we further consider the position of  the vertex $v_i$ in $T$. Clearly, $d_G(v_i,v_k)=3\leq d_T(v_i,v_k)$. If $d_T(v_k,v_i)>3$, then, from $d_G(x,v_i)=3$, we get $d_G(v_i,v_k)\geq 4$, which is a clear contradiction. Thus $d_T(v_k,v_i)=3$. If $xv_k\notin E(G)$,  then, assuming without loss of generality that $k<i$, we find that $d_G(v_k,v_{i+1})=4$, contradicting the $3$-ASC property of $G$. Next we deal with the subcase when $xv_k\in E(G)$. Note that $n\geq 10$.  There is at least one vertex $v_j$ in $T$ such that $d_T(v_i,v_j)\geq 4$. Then in this subcase we have $d_G(v_i,v_j)\geq 4$ whenever $v_jx\in E(G)$ or not. A contradiction comes again.

We have thus proved that ${\rm ecc}_G(v_n)=3$. Next we distinguish the following two subcases.

\medskip\noindent
\textbf{Subcase 2.1.} $xv_k\notin E(G)$. \\
In this subcase, we first prove ${\rm ecc}_G(x)=3$. If not, then we have ${\rm ecc}_G(x)=4$.  Then $x$ must lie on an induced cycle $C_8$ of $G$ where $v_i\in V(C_8)$ with $d_G(x,v_i)=4$, or on an induced $(x,v_i)$-path with $d_G(x,v_i)=4$ and $v_i$ being pendant in $G$. In the former case, we can find at least one vertex $v_j\in V(C_8)$ such that $d_G(v_n,v_j)\geq 4$. This is a clear contradiction. In the latter case, we claim that $v_i$ is the other diametrical vertex in $G$, that is, $d_G(v_i,v_s)\leq 3$ for any vertex $v_s$ with $1\leq s\leq n$. But this cannot be obeyed for any vertex $v_s$ with $d_T(v_s,v_i)>3$ for $n\geq 10$.

Next we  claim that ${\rm ecc}_G(v_k)=3$ and  assume that on the contrary ${\rm ecc}_G(v_k)=4$. Then there is exactly one vertex, say $v_i$, with $d_G(v_i,v_k)=4$. Without loss of generality, assume that $i>k$. Evidently, we have $d_T(v_i,v_k)\geq 4$ and $xv_i\notin E(G)$. Considering that ${\rm ecc}_G(v_n)={\rm ecc}_G(x)=3$, from the structure of $G$, we  observe that $d_G(x,v_i)=2$. From the fact that $d_G(v_i,v_j)\leq 3$ and $d_G(v_k,v_j)\leq 3$ for any vertex $v_j\in V(T)\setminus\{v_i,v_k\}$, we find that any vertex $v_j$ with $j\in\{m:\ m\leq k-1,\ \mbox{or}~ m\geq i+1\}$ must be adjacent to $x$ in $G$.

If $d_T(v_i,v_k)=4$, then $i=k+4$. Clearly, $xv_{k+2}\notin E(G)$. If not, we have ${\rm ecc}_G(x)=2$, which is impossible. Since $d_G(v_n,v_{k+3})\leq 3$, we get $xv_{k+3}\in E(G)$ from the fact that neither of $v_{k+2}$ and $v_{k+4}$ is adjacent to $x$ in $G$. Moreover, we have $xv_{k+1}\notin E(G)$. Otherwise, we get ${\rm ecc}_G(x)=2$, which is impossible. Then $v_{k+1}$ is just the eccentric vertex of $x$ in $G$.  Thus there must be a vertex $v_j$ $(j\neq k+3,k+4)$ with $d_T(v_k,v_j)\geq 3$ such that $d_G(v_{k+1},v_j)=4$, a contradiction.

If $d_T(v_i,v_k)>4$, we have  $xv_j\in E(G)$ for any vertex $v_j$ with $k+3<j<i$, since $d_G(v_k,v_j)\leq 3$.  Moreover, neither of $v_{k+1}$ and $v_{k+2}$ is adjacent to $x$ in $G$,  for otherwise ${\rm ecc}_G(x)=2$ would hold. But then $d_G(v_i,v_{k+1})\geq 4$, a contradiction.

 We have thus proved that ${\rm ecc}_G(x)={\rm ecc}_G(v_k)=3$. From ${\rm ecc}_G(v_k)=3$ we find that $xv_j\in E(G)$ for any vertex $v_j$ with $d_T(v_j,v_k)\geq 4$.  Because ${\rm ecc}_G(x)=3$, we infer that there exist three consecutive vertices, say $v_{p-1},v_p,v_{p+1}$,  such that they are at distance at most $3$ to $v_k$ and that none of them is adjacent to $x$ in $G$. It implies that $d_G(x,v_p)=3$. Since $n\geq 10$, there are at least two vertices $v_j$ with $d_G(v_j,v_p)=4$, contradicting to the $3$-ASC property of $G$.

\medskip\noindent
\textbf{Subcase 2.2.} $xv_k\in E(G)$.\\
In this subcase we first assume that ${\rm ecc}_G(x)=3$. Then there exists a vertex $v_i$ with $d_G(x,v_i)=3$. We may again assume without loss of generality that $i>k$. If $v_i$ is pendant in $G$, then there are at least two vertices at distance at least $4$ from $v_i$ in $G$, contradicting to the $3$-ASC property of $G$. If $v_i$ is non-pendant in $G$, then there are three consecutive vertices $v_{i-1}$, $v_i$, $v_{i+1}$, such that none of them is adjacent to $x$ in $G$ and $d_G(x,v_i)=3$ with $d_T(v_i,v_k)\geq 2$.  If $d_T(v_k,v_i)\geq 3$, we find that $d_G(v_n,v_i)=4$, which is impossible because of ${\rm ecc}_G(v_n)=3$. While $d_T(v_i,v_k)=2$, for $n\geq 10$, there exist at least two vertices $v_p$ with $d_T(v_i,v_p)\geq 4$ such that $d_G(v_i,v_p)\geq 4$. This contradicts to the $3$-ASC property of $G$. So it follows that ${\rm ecc}_G(x)=4$.

If ${\rm ecc}_G(x)=4$, then there is a vertex $v_p$ with $d_G(x,v_p)=4$. Obviously, $d_T(v_k,v_p)\geq 3$. Considering that $v_n,v_k,x$  form a triangle in $G$, we  get $d_G(v_n,v_p)\geq 4$,  the final contradiction.
\end{proof}

 Therefore Lemmas~\ref{Le3.3} and~\ref{Le3.4} immediately imply:

\begin{theorem}\label{thm:trees}
If $T$ is a tree of order $n\geq 10$ and diameter $n-2$, then $\theta_3(T)=2$.
\end{theorem}

In Theorem~\ref{thm:trees} we did not consider trees of small order. Note that $K_{1,3}$ is the unique tree of the smallest order of interest (that is, of order $4$). For $n=5$ there is also a unique tree of interest, while for $n=6$ there are two non-isomorphic trees of diameter $4$. Any of these trees is an induced subgraph of $C_6^*$ which is in turn a smallest $3$-ASC graph, hence we have the $3$-ASC index for $4\le n\le 6$.  However, for $7\le n\le 9$ the variety of the trees of interest becomes larger, hence determining the $3$-ASC index for all of them would be too extensive to be done here.

\section{Three open problems}
\label{sec:problems}

 We conclude the paper with three open problems. First, in view of Theorems~\ref{thm:constuction} and \ref{index-2r} we pose the following very ambitious task:

\begin{problem}
\label{prob:characterize}
For any given integer $k\in [1, 5]$ characterize the graphs with $\theta_3(G) = k$.
\end{problem}

 Next, in view of Lemma \ref{lem:3-or-4} and Theorems \ref{thm:new-added}-\ref{thm:union-of-complete}, it is natural to pose:

\begin{problem}
\label{prob:3-or-4}
Characterize the graphs of diameter $2$ with the $3$-ASC index equal to $3$ (equivalently, with the $3$-ASC index equal to $4$).
\end{problem}

A classical result from random graph theory asserts that almost any graph has diameter $2$, cf.~\cite[p.~312, Exercise~7]{diestel-2006}. Since by Lemma~\ref{lem:3-or-4} the $3$-ASC index of such a graph is either $3$ or $4$, our third problem reads as follows.

\begin{problem}
\label{prob:almost-all}
Is it true that the $3$-ASC index of almost every graphs is $3$ (resp.\ $4$)?
\end{problem}

\section*{Acknowledgements}

 K. X. is supported by NNSF of China (No. 11671202), China Postdoctoral Science Foundation (2013M530253, 2014T70512). K. C. D. is supported by National Research Foundation funded by the Korean government with the grant no. 2013R1A1A2009341. S. K.  acknowledges the financial support from the Slovenian Research Agency (research core funding No.\ P1-0297).


\begin{thebibliography}{99}

\bibitem{alcon-2017}
  L.~Alc{\'o}n, F.~Bonomo, M. P. Mazzoleni,
  Vertex intersection graphs of paths on a grid: characterization within block graphs,
  Graphs Combin.\ 33 (2017) 653--664.

\bibitem{BBCK2012}
 K. Balakrishnan, B. Bre\v{s}ar, M. Changat, S. Klav\v{z}ar, I. Peterin, A. R. Subhamathi,
  Almost self-centered median and chordal graphs,
  Taiwanese J.\ Math.\ 16 (2012) 1911--1922.

\bibitem{buckley-1989}
  F. Buckley, Self-centered graphs,
  Graph Theory and Its Applications: East and West,
  Ann.\ New York Acad.\ Sci.\  576 (1989) 71--78.

\bibitem{cheng-2014}
  Y-K. Cheng, L-Y. Kang, H. Yan,
  The backup 2-median problem on block graphs,
  Acta Math.\ Appl.\ Sin.\ Engl.\ Ser.\ 30 (2014) 309--320.

\bibitem{DLG2013}
  K. C. Das, D.W. Lee, A. Graovac,
  Some properties of the Zagreb eccentricity indices,
  Ars Math.\ Contemp.\ 6 (2013) 117--125.

\bibitem{DN2016}
  K. C. Das, M. J. Nadjafi-Arani,
  On maximum Wiener index of trees and graphs with given radius,
  J.\ Comb.\ Optim.\ 34 (2017) 574--587.

\bibitem{diestel-2006}
  R. Diestel,
  Graph Theory,
  Springer-Verlag, Berlin, 2006.

\bibitem{Gu-1993}
  W. Gu,
  On minimal embedding of two graphs as center and periphery,
  Graphs Combin. 9 (1993) 315--323.

\bibitem{hong-2014}
  Y. Hong, L. Kang,
  Backup 2-center on interval graphs,
  Theoret.\ Comput.\ Sci.\ 445 (2012) 25--35.

\bibitem{hl-00}
  T. C. Huang, J-C. Lin,  H-J. Chen,
  A self-stabilizing algorithm which finds a 2-center of a tree,
  Comput.\ Math.\ Appl.\ 40 (2000) 607--624.

\bibitem{IYF2011}
 A.  Ili\'{c}, G. Yu, L. Feng,
  On the eccentric distance sum of graphs,
  J.\ Math.\ Anal.\ Appl.\ 381 (2011) 590--600.

\bibitem{KLSX-2017+}
  S.  Klav\v zar, H. Liu, P. Singh, K. Xu,
  Constructing almost peripheral and almost self-centered graphs revisited,
  Taiwanese J.\ Math.,  21 (2017) 705--717.

\bibitem{KNW2011}
 S. Klav\v{z}ar, K. P. Narayankar, H. B. Walikar,
  Almost self-centered graphs,
  Acta Math.\ Sin.\ (Engl.\ Ser.) 27 (2011) 2343--2350.

\bibitem{KNWL2014}
  S.  Klav\v{z}ar, K. P. Narayankar, H. B. Walikar, S. B. Lokesh,
  Almost-peripheral graphs,
  Taiwanese J.\ Math.\ 18 (2014) 463--471.

\bibitem{krnc-2015}
  M.  Krnc, R. \v{S}krekovski,
   Group centralization of network indices,
   Discrete Appl.\ Math.\ 186 (2015) 147--157.

\bibitem{madalloni-2016}
 A.  Maddaloni, C. T. Zamfirescu,
  A cut locus for finite graphs and the farthest point mapping,
  Discrete Math.\ 339 (2016) 354--364.

\bibitem{MW-2000}
  D. Mubayi, D. B. West,
  On the number of vertices with specified eccentricity,
  Graphs Combin.\ 16 (2000) 441--452.

\bibitem{puerto-2008}
  J.  Puerto, A. Tamir, J. A. Mesa, D. P{\'e}rez-Brito,
   Center location problems on tree graphs with subtree-shaped customers,
   Discrete Appl.\ Math.\ 156 (2008) 2890--2910.

\bibitem{SGM1997}
  V. Sharma, R. Goswami, A. K. Madan,
  Eccentric connectivity index: A novel highly discriminating topological descriptor for structure-property and structure-activity studies,
  J.\ Chem.\ Inf.\ Comput.\ Sci.\ 37 (1997) 273--282.

\bibitem{ww-09}
 H-L. Wang,  B. Wu, K-M. Chao,
  The backup 2-center and backup 2-median problems on trees,
  Networks 53 (2009) 39--49.

\bibitem{XDL2016}
  K. Xu, K. C. Das, H. Liu,
  Some extremal results on the connective eccentricity index of graphs,
  J.\ Math.\ Anal.\ Appl.\ 433 (2016) 803--817.

\bibitem{xu-2016+}
  K. Xu, K. C. Das, A. D. Maden,
  On a novel eccentricity-based invariant of a graph,
  Acta Math.\ Sin.\ (Engl.\ Ser.)  32 (2016) 1477--1493.

\bibitem{XLDGF2014}
  K. Xu, M. Liu, K. C. Das, I. Gutman, B. Furtula,
  A survey on graphs extremal with respect to distance-based topological indices,
  MATCH Commun.\ Math.\ Comput.\ Chem.\ 71 (2014) 461--508.

\bibitem{YQTF2014}
  G. Yu, H. Qu, L. Tang, L. Feng,
  On the connective eccentricity index of trees and unicyclic graphs with given diameter,
  J.\ Math.\ Anal.\ Appl.\ 420 (2014) 1776--1786.
\end{thebibliography}
\end{document}